\documentclass{amsart}
\usepackage{amsthm}
\usepackage{amsmath, amssymb, graphicx}
\usepackage{color}
\usepackage{pst-all}

\psset{unit=1mm}

\theoremstyle{plain}
\newtheorem{prop}{Proposition}[section]
\newtheorem{coro}[prop]{Corollary}
\newtheorem{lemm}[prop]{Lemma}
\newtheorem{thrm}[prop]{Theorem}
\newtheorem*{thrmA}{Theorem A}

\theoremstyle{remark}
\newtheorem{rema}[prop]{Remark}

\theoremstyle{definition}
\newtheorem{defi}[prop]{Definition}
\newtheorem{exam}[prop]{Example}

\numberwithin{equation}{section}

\newdimen\arrayitem
\def\arrayitem{3mm}
\newcommand\SM[1]{\hbox to \arrayitem{\hfil$#1$\hfil}}
\def\0{\hbox to \arrayitem{\hfil$\cdot$\hfil}}
\def\1{\hbox to \arrayitem{\hfil$1$\hfil}}
\def\mo{\hbox to \arrayitem{\hfil$-\!1$\hfil}}
\def\2{\hbox to \arrayitem{\hfil$2$\hfil}}
\def\3{\hbox to \arrayitem{\hfil$3$\hfil}}
\def\4{\hbox to \arrayitem{\hfil$4$\hfil}}
\def\5{\hbox to \arrayitem{\hfil$5$\hfil}}
\def\6{\hbox to \arrayitem{\hfil$6$\hfil}}
\def\7{\hbox to \arrayitem{\hfil$7$\hfil}}
\def\8{\hbox to \arrayitem{\hfil$8$\hfil}}

\renewcommand\aa{a} 
\renewcommand\AA{A} 
\newcommand\ad{\mathrm{ad}}
\newcommand\bb{b} 
\newcommand\BB{B}
\newcommand\before{\Vdash}
\newcommand\Before[1]{\mathrel{\,\Vdash_{\!#1}\,}}
\newcommand\beforesym{\vDash}
\newcommand\cc{c}
\newcommand\CC{C}

\newcommand\Col{\mathrm{Col}}
\newcommand\comp{\mathbin{\scriptstyle\circ}}
\newcommand\const{c}
\newcommand\dd{d}
\newcommand\DD{D}
\newcommand\der{\partial}
\renewcommand\div{\mathrel{\vert}}

\newcommand\eps{\varepsilon}
\newcommand\ff{f}

\let\ge=\geqslant

\newcommand\GG{G}

\newcommand\HH{H}
\newcommand\Hom{\operatorname{Hom}}
\newcommand\ii{i}
\newcommand\inv{^{-1}}
\newcounter{ITEM}
\newcommand\ITEM[1]{\setcounter{ITEM}{#1}\leavevmode\hbox{\rm(\roman{ITEM})}}

\newcommand\kk{k}
\let\le=\leqslant

\newcommand\mm{m}

\newcommand\MOD[2]{#1\ \mathrm{mod}\, #2}
\newcommand\MODD[2]{#1\, [#2]}
\newcommand\nn{n}
\newcommand\notdiv{\mathrel{\not{}\hspace{-0.1ex}\vert\hspace{0.1ex}}}

\newcommand\op{\mathbin{\triangleright}}
\newcommand\OP[1]{\mathrel{\triangleright_{\hspace{-0.2ex}#1}}}
\newcommand\ops{\hspace{0.3ex}{\op}\hspace{0.3ex}}
\newcommand\pdots{\hspace{0.2ex}{\cdot}{\cdot}{\cdot}\hspace{0.2ex}}
\newcommand\per{\pi}
\newcommand\phih{\widehat{\VR(2,0)\smash\phi}}
\newcommand\phit{\widetilde{\VR(2,0)\smash\phi}}
\newcommand\pn{N}
\newcommand\pp{p}
\newcommand\ppn{2^\nn{-}1} 
\newcommand\proj{\mathrm{pr}}
\newcommand\qq{q}

\newcommand\Rack{{\scriptscriptstyle\mathrm{R}}}
\newcommand\rr{r}

\newcommand\RRRR{\mathbb{R}}
\renewcommand\SS{S}
\newcommand\thres{\vartheta}
\renewcommand\tt{t}

\def\VR(#1,#2){\vrule width0pt height#1mm depth#2mm}
\newcommand\uu{u}
\newcommand\vv{v}

\newcommand\wdots{, ...\hspace{0.2ex},}

\newcommand\xx{x}
\newcommand\yy{y}
\newcommand\zz{z}
\newcommand\ZZ{Z}
\newcommand\ZZZZ{\mathbb{Z}}

\begin{document}\rm

\title{Two- and three-cocycles for Laver tables}

\author{Patrick DEHORNOY}
\address{Laboratoire de Math\'ematiques Nicolas Oresme, UMR 6139 CNRS, Universit\'e de Caen BP 5186, 14032 Caen Cedex, France}
\address{Laboratoire Preuves, Programmes, Syst\`emes, UMR 7126 CNRS, Universit\'e Paris-Diderot Case 7014, 75205 Paris Cedex 13, France}
\email{dehornoy@math.unicaen.fr}
\urladdr{//www.math.unicaen.fr/\!\hbox{$\sim$}dehornoy}

\author{Victoria LEBED}
\address{OCAMI, Osaka City University, Japan}
\email{lebed.victoria@gmail.com}
\urladdr{//www.math.jussieu.fr/\!\hbox{$\sim$}lebed}



\thanks{The second author was supported by a JSPS Fellowship. She would like to thank OCAMI and in particular Seiichi Kamada for hospitality.}

\maketitle

\begin{abstract}
We determine all $2$- and $3$-cocycles for Laver tables, an infinite sequence of finite structures obeying the left-selfdistributivity law; in particular, we describe simple explicit bases. This provides a number of new positive braid invariants and paves the way for further potential topological applications. An important tool for constructing a combinatorially meaningful basis of $2$-cocycles is the right-divisibility relation on Laver tables, which turns out to be a partial ordering.
\end{abstract}

Introduced by Richard Laver in~\cite{Lvd}, the Laver tables are an infinite series of finite structures~$(\AA_\nn, \OP\nn)$ where $\AA_\nn$ is the set $\{1, 2 \wdots 2^\nn\}$ and $\OP\nn$ is a binary operation on~$\AA_\nn$ that obeys the left-selfdistributivity law
\begin{equation}
\label{E:LD}
\tag{$LD$}
\xx \op (\yy \op \zz) = (\xx \op \yy) \op (\xx \op \zz)
\end{equation}
and the initial condition $\xx \op 1 = \MOD{\xx + 1}{2^\nn}$.
According to an approach that can be traced back to Joyce~\cite{Joy}, Matveev~\cite{Mat}, and Brieskorn~\cite{Bri}, selfdistributivity is an algebraic distillation of the Reidemeister move of type~$\mathrm{III}$ (or, briefly, Reidemeister~$\mathrm{III}$ move)---see Figure~\ref{F:Colouring}---and, therefore, it is not surprising that structures involving operations that obey the law~\eqref{E:LD} can often lead to powerful topological constructions. 

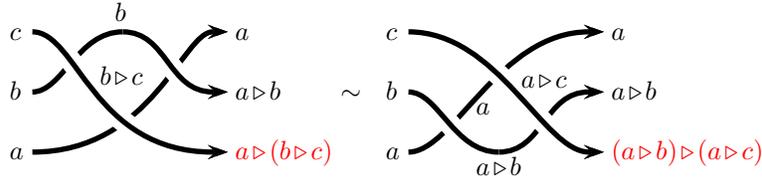
\begin{figure}[htb]
\begin{picture}(93,25)(0,0)
\psbezier[linewidth=2pt,border=3pt](0,8)(4,8)(6,16)(12,16)
\psbezier[linewidth=2pt,border=3pt](0,0)(16,0)(20,16)(24,16)
\psbezier[linewidth=2pt,border=3pt](12,16)(18,16)(18,8)(24,8)
\psbezier[linewidth=2pt,border=3pt](0,16)(6,16)(8,0)(24,0)
\psline[linewidth=2pt,border=3pt]{->}(24,0)(26,0)
\psline[linewidth=2pt,border=3pt]{->}(24,8)(26,8)
\psline[linewidth=2pt,border=3pt]{->}(24,16)(26,16)
\put(-3,-1){$\aa$}
\put(-3,7){$\bb$}
\put(-3,15){$\cc$}
\put(27,-1){\color{red} $\aa \ops (\bb \ops \cc)$}
\put(27,7){$\aa\ops\bb$}
\put(27,15){$\aa$}
\put(9,9){$\bb\ops\cc$}
\put(11,17.5){$\bb$}
\put(41,7){$\sim$}
\psbezier[linewidth=2pt,border=3pt](50,0)(56,0)(62,16)(74,16)
\psbezier[linewidth=2pt,border=3pt](62,0)(68,0)(68,8)(74,8)
\psbezier[linewidth=2pt,border=3pt](50,16)(62,16)(68,0)(74,0)
\psbezier[linewidth=2pt,border=3pt](50,8)(54,8)(56,0)(62,0)
\psline[linewidth=2pt,border=3pt]{->}(74,0)(76,0)
\psline[linewidth=2pt,border=3pt]{->}(74,8)(76,8)
\psline[linewidth=2pt,border=3pt]{->}(74,16)(76,16)
\put(47,-1){$\aa$}
\put(47,7){$\bb$}
\put(47,15){$\cc$}
\put(77,-1){\color{red} $(\aa \ops\bb) \ops (\aa \ops \cc)$}
\put(77,7){$\aa\ops\bb$}
\put(77,15){$\aa$}
\put(59,-3){$\aa\ops\bb$}
\put(59,5){$\aa$}
\put(65,8.5){$\aa\ops\cc$}
\end{picture}
\caption{\sf\small Translation of invariance under Reidemeister~$\mathrm{III}$ move into the language of selfdistributivity: when colours from a set~$\SS$ are put on the left ends of the strands and then propagated so that a $\bb$-coloured strand becomes $\aa \ops \bb$-coloured when it overcrosses an $\aa$-coloured arc, then the output colours are invariant under Reidemeister~$\mathrm{III}$ move if and only if the operation~$\ops$ obeys the left-selfdistributivity law.}
\label{F:Colouring}
\end{figure}

In practice, one of the most fruitful methods so far consists in developing a (co)homological appoach, as explained for instance in the (very accessible) surveys~\cite{Kam,CarterSurvey}. More specifically, according to schemes that will be recalled in Section~\ref{S:RackHom} below, every $2$- or $3$-cocycle for a selfdistributive structure~$\SS$ leads to a positive braid invariant.  Inserting Reidemeister II and I  moves into the picture leads to considering particular selfdistributive structures, namely racks, which are selfdistributive structures in which all left-translations are bijective~\cite{FeR}, and quandles, which are idempotent racks~\cite{Joy,Mat}. Racks or quandles can thus be used to construct not only positive braid invariants but also general braid and, respectively, knot and link invariants.

Laver tables are selfdistributive structures that are not racks, hence \textit{a fortiori} not quandles, but techniques similar to those developed by the first author in the case of free selfdistributive structures~\cite{Dgd} might make it possible to use them in topology and, due to their fundamental position among selfdistributive structures (see~\cite{Dgd,DraAlg, DraGro} for details), it is reasonable to expect promising developments. In such a context, the obvious first step in the direction of possible topological applications of Laver tables is to analyze the associated $2$- and $3$-cocycles. What we do in this paper is to provide an exhaustive description of all such cocycles:

\begin{thrmA}
\ITEM1 For every~$\nn \ge 0$, the $\ZZZZ$-valued $2$-cocycles for~$\AA_\nn$ make a free $\ZZZZ$-module of rank~$2^\nn$, with a basis consisting of the constant cocycle and of $2^\nn -1$ explicit $\{0,1\}$-valued coboundaries defined for $1 \le \qq < 2^\nn$ by
$$\psi_{\qq, \nn}(\xx, \yy) = \begin{cases}
\ 1&\mbox{if $\qq$ can be written as $\pp \OP\nn \yy$ for some~$\pp$, but not as $\rr \OP\nn(\xx \OP\nn \yy)$},\\
\ 0&\mbox{otherwise}.\end{cases}$$

\ITEM2 For every~$\nn \ge 0$, the $\ZZZZ$-valued $3$-cocycles for~$\AA_\nn$ make a free $\ZZZZ$-module of rank~$2^{2\nn}-2^\nn+1$, with a basis consisting of the constant cocycle and of $2^{2\nn}-2^\nn$ explicit $\{0,\pm 1\}$-valued coboundaries indexed by pairs~$(\pp, \qq)$ with $1 \le \pp, \qq \le 2^\nn$ and $\pp \neq 2^\nn-1$.
\end{thrmA}

Let us mention that the above results still hold without change (and at no extra cost) when $\ZZZZ$ is replaced with any abelian group~$\GG$. 

 Theorem~A shows that the families of $2$- and $3$-cocycles on the Laver tables are quite rich,  being essentially indexed by the elements of~$\AA_\nn$ and the pairs of elements of~$\AA_\nn$, respectively. Moreover, we shall see that these cocycles capture a number of deep phenomena connected with Laver tables, which shows that they are (highly) nontrivial. For instance,  periods and thresholds, two series of parameters that  witness  the combinatorial complexity of the tables, can be recovered from $2$-cocycles.  Similarly, the cocycles~$\psi_{\qq, \nn}$ of Theorem~A inherit the order properties provided by the right-divisibility relation of~$\AA_\nn$, a certain partial ordering whose properties remain at the moment largely unknown. Alltogether, these elements may appear as  a favourable sign for potential applications. In particular,  having an explicit basis of $2$-cocycles made of functions  with nonnegative values  seems especially promising in view of combinatorial interpretations, typically for counting arguments.  Also, independently of any further development, it should be remembered that, according to the principles recalled in Lemmas~\ref{L:Use2Cocycle} and~\ref{L:Use3Cocycle} below, every explicit cocycle we describe directly gives rise to a new positive braid invariant.

The paper is organized as follows.  Section~\ref{S:Laver} contains a short, self-contained introduction to the Laver tables. Next, we recall in Section~\ref{S:RackHom} the basic notions of rack homology as well as the principle for using $2$- and $3$-cocycles to construct topological invariants. In Section~\ref{S:TwoPhi}, we construct a first basis for the $2$-cocycles for~$\AA_\nn$. In Section~\ref{S:Poset}, we investigate the right-divisibility relation of Laver tables, a digression of independent interest, which is then used in Section~\ref{S:TwoPsi} to construct a second basis for the $2$-cocycles for~$\AA_\nn$, thus completing the proof of Point~\ITEM1 in Theorem~A. Finally, in Section~\ref{S:Three}, we similarly analyze $3$-cocycles and establish Point~\ITEM2 in Theorem~A.

\section{Laver tables: construction and properties}
\label{S:Laver}

Laver tables are the elements of an infinite family of selfdistributive structures discovered by Richard Laver around 1995 as a by-product of his analysis of iterations of elementary embeddings in Set Theory~\cite{Lvd}. Their existence and characterization are specified in the following result. Here and everywhere in the sequel, we write $\MOD\pp\mm$ for the unique integer in $\{1 \wdots \mm\}$ that is equal to~$\pp$ \textit{modulo}~$\mm$. 

\begin{thrm}[Laver, \cite{Dgd}] 
\label{T:Laver} 
\ITEM1 For every $\nn \ge 0$, there exists a unique binary operation~$\op$ on the set $\{1 \wdots 2^\nn\}$ obeying the laws
\begin{gather}
\xx \op 1 = \MOD{\xx + 1}{2^\nn},
\label{E:Laver1}\\
\xx \op (\yy \op 1) = (\xx \op \yy) \op (\xx \op 1);
\label{E:Laver2}
\end{gather}
the operation~$\op$ obeys the left-selfdistributivity law~$(LD)$. 

\ITEM2 For every~$\pp \le 2^\nn$, there exists a (unique) integer~$2^\rr$ satisfying 
\begin{equation*}
\pp \op 1 < \pp \op 2 < \pdots < \pp \op 2^\rr = 2^\nn,
\end{equation*}
and the subsequent values $\pp \op \qq$ then repeat periodically. 
\end{thrm}

\begin{defi}
For $\nn \ge 0$, the structure with domain $\{1 \wdots 2^\nn\}$ specified in Theorem~\ref{T:Laver} is called the \emph{$\nn$th Laver table} and it is denoted by~$\AA_\nn$. The number~$2^\rr$ in~\ITEM2 is called the \emph{period} of~$\pp$ in~$\AA_\nn$ and it is denoted by~$\per_\nn(\pp)$.
\end{defi}

To avoid ambiguity, especially when several Laver tables are considered simultaneously, we shall denote by~$\OP\nn$  the operation of~$\AA_\nn$. 

Explicitly computing Laver tables is easy: identifying the structure with a table where the $(\xx,\yy)$-entry contains $\xx \op \yy$, and starting from an empty table, \eqref{E:Laver1} prescribes the values of the products in the first column and then \eqref{E:Laver2} enables one to inductively complete the last row from left to right, then the penultimate row, \textit{etc.}, finishing with the first, which is computed last. Theorem~\ref{T:Laver}\ITEM2 says that every row in a Laver table is periodic, with a period that is a power of~$2$.
The first tables are displayed in Table~\ref{T:First}.

\begin{table}[htb]
\begin{gather*}
\begin{tabular}{c|c}
$\AA_0$&$1$\\
\hline
$1$&$1$
\end{tabular}
\quad 
\begin{tabular}{c|cc}
$\AA_1$&$1$&$2$\\
\hline
$1$&$2$&$2$\\
$2$&$1$&$2$\\
\end{tabular}
\quad 
\begin{tabular}{c|cccc}
$\AA_2$&$1$&$2$&$3$&$4$\\
\hline
$1$&$2$&$4$&$2$&$4$\\
$2$&$3$&$4$&$3$&$4$\\
$3$&$4$&$4$&$4$&$4$\\
$4$&$1$&$2$&$3$&$4$\\
\end{tabular}
\quad 
\begin{tabular}{c|cccccccc}
$\AA_3$&$1$&$2$&$3$&$4$&$5$&$6$&$7$&$8$\\
\hline
$1$&$2$&$4$&$6$&$8$&$2$&$4$&$6$&$8$\\
$2$&$3$&$4$&$7$&$8$&$3$&$4$&$7$&$8$\\
$3$&$4$&$8$&$4$&$8$&$4$&$8$&$4$&$8$\\
$4$&$5$&$6$&$7$&$8$&$5$&$6$&$7$&$8$\\
$5$&$6$&$8$&$6$&$8$&$6$&$8$&$6$&$8$\\
$6$&$7$&$8$&$7$&$8$&$7$&$8$&$7$&$8$\\
$7$&$8$&$8$&$8$&$8$&$8$&$8$&$8$&$8$\\
$8$&$1$&$2$&$3$&$4$&$5$&$6$&$7$&$8$\\
\end{tabular}\\
\small\def\arrayitem{5mm}\begin{tabular}{c|cccccccccccccccc}
\VR(0,1.5)$\AA_4$&\SM1\SM2\SM3\SM4\SM5\SM6\SM7\SM8\SM9\SM{10}\SM{11}\SM{12}\SM{13}\SM{14}\SM{15}\SM{16}\\
\hline
\VR(3,0)\SM1&\SM2\SM{12}\SM{14}\SM{16}\SM2\SM{12}\SM{14}\SM{16}\SM2\SM{12}\SM{14}\SM{16}\SM2\SM{12}\SM{14}\SM{16}\\
\SM2&\SM3\SM{12}\SM{15}\SM{16}\SM3\SM{12}\SM{15}\SM{16}\SM3\SM{12}\SM{15}\SM{16}\SM3\SM{12}\SM{15}\SM{16}\\
\SM3&\SM4\SM8\SM{12}\SM{16}\SM4\SM8\SM{12}\SM{16}\SM4\SM8\SM{12}\SM{16}\SM4\SM8\SM{12}\SM{16}\\
\SM4&\SM5\SM6\SM7\SM8\SM{13}\SM{14}\SM{15}\SM{16}\SM5\SM6\SM7\SM8\SM{13}\SM{14}\SM{15}\SM{16}\\
\SM5&\SM6\SM8\SM{14}\SM{16}\SM6\SM8\SM{14}\SM{16}\SM6\SM8\SM{14}\SM{16}\SM6\SM8\SM{14}\SM{16}\\
\SM6&\SM7\SM8\SM{15}\SM{16}\SM7\SM8\SM{15}\SM{16}\SM7\SM8\SM{15}\SM{16}\SM7\SM8\SM{15}\SM{16}\\
\SM7&\SM8\SM{16}\SM8\SM{16}\SM8\SM{16}\SM8\SM{16}\SM8\SM{16}\SM8\SM{16}\SM8\SM{16}\SM8\SM{16}\\
\SM8&\SM9\SM{10}\SM{11}\SM{12}\SM{13}\SM{14}\SM{15}\SM{16}\SM9\SM{10}\SM{11}\SM{12}\SM{13}\SM{14}\SM{15}\SM{16}\\
\SM9&\SM{10}\SM{12}\SM{14}\SM{16}\SM{10}\SM{12}\SM{14}\SM{16}\SM{10}\SM{12}\SM{14}\SM{16}\SM{10}\SM{12}\SM{14}\SM{16}\\
\SM{10}&\SM{11}\SM{12}\SM{15}\SM{16}\SM{11}\SM{12}\SM{15}\SM{16}\SM{11}\SM{12}\SM{15}\SM{16}\SM{11}\SM{12}\SM{15}\SM{16}\\
\SM{11}&\SM{12}\SM{16}\SM{12}\SM{16}\SM{12}\SM{16}\SM{12}\SM{16}\SM{12}\SM{16}\SM{12}\SM{16}\SM{12}\SM{16}\SM{12}\SM{16}\\
\SM{12}&\SM{13}\SM{14}\SM{15}\SM{16}\SM{13}\SM{14}\SM{15}\SM{16}\SM{13}\SM{14}\SM{15}\SM{16}\SM{13}\SM{14}\SM{15}\SM{16}\\
\SM{13}&\SM{14}\SM{16}\SM{14}\SM{16}\SM{14}\SM{16}\SM{14}\SM{16}\SM{14}\SM{16}\SM{14}\SM{16}\SM{14}\SM{16}\SM{14}\SM{16}\\
\SM{14}&\SM{15}\SM{16}\SM{15}\SM{16}\SM{15}\SM{16}\SM{15}\SM{16}\SM{15}\SM{16}\SM{15}\SM{16}\SM{15}\SM{16}\SM{15}\SM{16}\\
\SM{15}&\SM{16}\SM{16}\SM{16}\SM{16}\SM{16}\SM{16}\SM{16}\SM{16}\SM{16}\SM{16}\SM{16}\SM{16}\SM{16}\SM{16}\SM{16}\SM{16}\\
\SM{16}&\SM1\SM2\SM3\SM4\SM5\SM6\SM7\SM8\SM9\SM{10}\SM{11}\SM{12}\SM{13}\SM{14}\SM{15}\SM{16}
\end{tabular}
\end{gather*}
\caption{\sf\small The first five Laver tables; observe the periodic behaviour of the rows as predicted by Theorem~\ref{T:Laver}\ITEM2: for instance, we read the values $\per_0(1) = \per_1(1) = 1$, $\per_2(1) = 2$, $\per_3(1) = \per_4(1) = 4$.}
\label{T:First}
\end{table}

For further reference, let us note some consequences of the facts mentioned in Theorem~\ref{T:Laver}. First, as the rows of~$\AA_\nn$ are of length~$2^\nn$ and every period~$\per_\nn(\pp)$ is a power of~$2$, hence a divisor of~$2^\nn$, we deduce for every~$\pp$ in~$\{1 \wdots 2^\nn\}$ the equality
\begin{equation}
\label{E:LastColumn}
\pp \OP\nn 2^\nn = 2^\nn.
\end{equation}
We also note that $\pp < 2^\nn$ implies for every~$\qq$
\begin{equation}
\label{E:LeftMultiples}
\pp < \pp  \OP\nn \qq,  
\end{equation}
and that 
\begin{equation}
\label{E:PlusOne}
\pp \OP\nn \qq = \pp \OP\nn \qq' \mbox{\quad implies \quad} \pp \OP\nn (\MOD{\qq+1}{2^\nn}) = \pp \OP\nn (\MOD{\qq'+1}{2^\nn}):
\end{equation}
by~\eqref{E:Laver1}, the latter equality is $\pp \OP\nn (\qq \OP\nn 1) = \pp \OP\nn (\qq' \OP\nn 1)$, hence equivalently $(\pp \OP\nn \qq) \OP\nn (\pp \OP\nn 1) = (\pp \OP\nn \qq') \OP\nn (\pp \OP\nn 1)$ owing to~\eqref{E:Laver2}.

On the other hand, a direct verification gives, for every~$\qq$, the values
\begin{equation}
\label{E:LastRows}
(2^\nn{-}1) \OP\nn \qq = 2^\nn \mbox{\quad and \quad} 2^\nn \OP\nn \qq = \qq:
\end{equation}
 in other words, we have  $\per_\nn(2^\nn{-}1) = 1$ and $\per_\nn(2^\nn) = 2^\nn$.

By construction, the Laver table~$\AA_\nn$ is generated by the element~$1$, and it is therefore an example of a monogenerated left-selfdistributive structure. It turns out that every monogenerated left-selfdistributive structure can be obtained from Laver tables using some simple operations~\cite{DraAlg, DraGro}. Note that such structures are very far from the most common selfdistributive structures, such as groups equipped with the conjugacy operation $\xx \op \yy = \xx\yy\xx\inv$ or, more generally, $\xx \op \yy = \xx \ff(\yy\xx\inv)$ with $\ff$ an endomorphism of the considered group. Contrary to the latter examples, Laver tables (except~$\AA_0$) are not racks, since for most values of~$\pp$ the map $\qq \mapsto \pp \OP\nn \qq$ is not bijective. Moreover, they do not obey the law $(\xx \op \xx) \op \yy = \xx \op \yy$, which is obeyed in every rack. 

Laver tables are strongly connected to one another. Indeed, there exists a natural projection from~$\AA_\nn$ to~$\AA_{\nn-1}$ and, in the other direction, constructing the rows of~$\pp$ and $p + 2^{\nn-1}$ in~$\AA_\nn$ from the row of~$\pp$ in~$\AA_{\nn-1}$ requires determining one single integer parameter.

\begin{prop}[Laver, \cite{Dgd}]
\label{P:Proj}
\ITEM1 For every~$\nn$, the map~$\proj_\nn : \pp \mapsto \MOD\pp{2^{\nn-1}}$ defines a surjective homomorphism from~$\AA_\nn$ to~$\AA_{\nn-1}$. 

\ITEM2 For all~$\nn$ and $\pp \le 2^{\nn-1}$, there exists a number $\rr$ with $0 \le \rr \le \per_{\nn-1}(\pp)$, such that, for every~$\qq \le \per_{\nn}(\pp)$, one has
\begin{equation}
\label{E:Proj}
\pp \OP\nn \qq = \begin{cases}
\pp \OP{\nn-1} \qq &\mbox{for $1 \le \qq \le \rr$},\\
\pp \OP{\nn-1} \qq + 2^{\nn-1} &\mbox{for $\rr < \qq \le \per_{\nn-1}(\pp)$}.
\end{cases}
\end{equation}
\end{prop}

\begin{defi}
The number~$\rr$ in~\ITEM2 is called the \emph{threshold} of~$\pp$ in~$\AA_\nn$ and it is denoted by~$\thres_\nn(\pp)$.
\end{defi}

Proposition~\ref{P:Proj} implies that the structure~$\AA_{\nn-1}$ and the sequence of numbers $\thres_\nn(1)$, ..., $\thres_\nn(2^{\nn-1}-1)$ completely determine the structure~$\AA_\nn$. 
Indeed, Point~\ITEM1 implies that, for all $\pp, \qq \le 2^\nn$, if we write $\bar\pp$ for $\MOD\pp{2^{\nn-1}}$ and $\bar\qq$ for~$\MOD\qq{2^{\nn-1}}$, then we have $\proj_\nn(\pp \op_\nn \qq) = \bar\pp \OP{\nn-1} \bar\qq$,  whence 
$$\pp \op_\nn \qq \ \in\  \{ \: \bar\pp \OP{\nn-1} \bar\qq \: , \:  \bar\pp \OP{\nn-1} \bar\qq + 2^{\nn-1} \: \}. $$
 For $2^{\nn-1} \le \pp < 2^\nn$, \eqref{E:LeftMultiples} implies $2^{\nn-1} \le \pp < \pp \OP\nn \qq$, so the only possibility is $\pp \OP\nn \qq = \bar\pp \OP{\nn-1} \bar\qq + 2^\nn$.
By contrast, for $\pp < 2^{\nn-1}$, neither value is excluded and, assuming $\nn \ge 1$, we have for instance $\pp \OP\nn 1 = \pp \OP{\nn-1} 1 = \pp + 1$. What Point~\ITEM2 of Proposition~\ref{P:Proj} says is that the row of~$\pp$ in~$\AA_\nn$ begins with~$\thres_\nn(\pp)$ values common with those of~$\AA_{\nn-1}$, followed by~$\per_\nn(\pp)-\thres_\nn(\pp)$ values that are shifted by~$2^{\nn-1}$. In order to better understand the situation, consider two cases: if we have $\thres_\nn(\pp) < \per_{\nn-1}(\pp)$, then we find $\pp \OP\nn \per_{\nn-1}(\pp) = 2^\nn$, in which case we deduce $\per_\nn(\pp) = \per_{\nn-1}(\pp)$; if we have $\thres_\nn(\pp) = \per_{\nn-1}(\pp)$, then we find $\pp \OP\nn \per_{\nn-1}(\pp) = 2^{\nn-1}$, and, as the values in the row of~$\pp$ in~$\AA_\nn$ must increase until the value~$2^\nn$ occurs, we deduce $\per_\nn(\pp) = 2\per_{\nn-1}(\pp)$, and the description of the $\pp$th row of~$\AA_\nn$ given in~\eqref{E:Proj} can be completed with
$$\pp \OP\nn \qq = \pp \OP{\nn-1} (\qq {-} \per_{\nn-1}(\pp)) + 2^{\nn-1} \mbox{\qquad for }\per_{\nn-1}(\pp) < \qq \le 2\per_{\nn-1}(\pp).$$
For instance, one can check in Table~\ref{T:First} that the threshold~$\thres_{3}(1)$ is~$2$ whereas $\thres_{4}(1)$ is~$1$: in the first case, this means that the first two values in the first row of~$\AA_{3}$, namely~$2$ and~$4$, are the first two values in the first row of~$\AA_{2}$, so, as $2$ is the period of~$1$ in~$\AA_2$, the period of~$1$ must jump to~$4$ in~$\AA_3$, the values necessarily being $2, 4, 6, 8$; in the second case, this means that only the first value~$2$ occurs in~$\AA_4$, the three other values being shifted by~$8$, so the period of~$1$ remains~$4$ in~$\AA_4$, the values in the first row necessarily being~$2$, $12$ (= $4{+}8$), $14$ ($= 6{+}8$), and~$16$.

Composing the projections~$\proj_\nn$ provides for all $\nn \ge \mm$ a surjective homomorphism~$\proj_{\nn, \mm}$ from~$\AA_\nn$ to~$\AA_\mm$. Some properties can then be lifted from~$\AA_\mm$ to~$\AA_\nn$. A typical example is as follows.

\begin{coro}
\label{C:Odd}
For all~$\nn \ge 1$ and $\pp, \qq \le 2^\nn$, the value~$\pp \OP\nn \qq$ is odd if and only if $\pp$ is even and $\qq$ is odd.
\end{coro}

\begin{proof}
Write $\bar\pp = \proj_{\nn, 1}(\pp)$ and $\bar\qq = \proj_{\nn, 1}(\qq)$. Proposition~\ref{P:Proj} implies $\proj_{\nn, 1}(\pp \OP\nn \qq) = \bar\pp \op_1 \bar\qq$. Now, we see on Table~\ref{T:First} that, for all $\bar\pp, \bar\qq$ in~$\{1, 2\}$, the value $\bar\pp \op_1 \bar\qq$ is odd if and only if we have $\bar\pp = 2$ and $\bar\qq = 1$, that is, if $\pp$ is even and $\qq$ is odd.
\end{proof}

Observe that the limit of the inverse system associated with the projections~$\proj_{\nn, \mm}$  consists of a selfdistributive operation on the set~$\ZZZZ_2$ of $2$-adic numbers.

To conclude this brief introduction to Laver tables, let us mention that some of their combinatorial properties, such as the result that the period~$\per_\nn(1)$ tends to~Ê$\infty$ with~$\nn$, keep so far an unusual logical status, being known to follow from some unprovable large cardinal axiom, but remaining open when such an axiom is not assumed~\cite{Dfp}. This paradoxical situation comes from the possibility of investigating Laver tables using elementary embeddings when the latter exist. However, all subsequent developments in the current article are independent from such issues.

\section{Basics on rack homology}
\label{S:RackHom}

In order to place our two-cocycle calculations for Laver tables in an appropriate context, we start with recalling some generalities on the rack homology of left-selfdistributive structures, originating from \cite{FRSRackHom}. Note that, working with left- and not with right-selfdistributivity as in most topological papers, we use a symmetric version of their constructions.

\begin{prop}
\label{P:Bicomplex}\cite{FRSRackHom}
Assume that $\op$ is a binary operation on a set~$\SS$ that obeys the left-selfdistributive law. For~$\kk \ge 1$, let $\CC_\kk(\SS)$ be a free $\ZZZZ$-module based on~$\SS^{\kk}$, and put $\CC_0(\SS)=\ZZZZ$. For $\kk > 0$ and $1 \le i \le \kk$, let $d^{\,\op}_{\kk;i}, d^{0}_{\kk;i} : \CC_\kk(\SS) \rightarrow \CC_{\kk-1}(\SS)$ be linear maps defined on~$\SS^{\kk}$ by 
\begin{gather*}
d^{\,\op}_{\kk;i}(x_1 \wdots  x_\kk) = (x_1 \wdots  x_{i-1}, \widehat{\xx_\ii}, x_i \op x_{i+1} \wdots  x_i \op x_{\kk}), \\
d^{\,0}_{\kk;i}(x_1 \wdots  x_\kk) = (x_1 \wdots  x_{i-1}, \widehat{x_i}, x_{i+1} \wdots x_{\kk}).
\end{gather*}
Put $\der^{\,\op}_{\kk}:= \sum_{i=1}^\kk (-1)^{i-1} d^{\,\op}_{\kk;i},$ and $\der^{\,0}_{\kk}:= \sum_{i=1}^\kk (-1)^{i-1} d^{0}_{\kk;i}$. Then $(\CC_\kk(\SS),\der^{\,\op}_{\kk},\der^{\,0}_{\kk})$ is a chain bicomplex, that is,  for every $\kk \ge 2$, we have 
\begin{equation}
\label{E:Bicomplex}
\der^{\,\op}_{\kk-1} \circ \der^{\,\op}_{\kk} = \der^{\,0}_{\kk-1} \circ \der^{\,0}_{\kk} = \der^{\,\op}_{\kk-1} \circ \der^{\,0}_{\kk}  + \der^{\,0}_{\kk-1} \circ \der^{\,\op}_{\kk}  = 0.
\end{equation}
\end{prop}

\begin{proof}
The fastest and probably the most conceptual proof of the statement is based on the (pre)cubical cohomology ideas, as developed in \cite{SerreThesis,KanCubic,BH_cubes}. Concretely, a direct computation shows that, for all $1 \le j < i \le \kk$ and for every choice of $\diamond$ and $\star$ in~$\{\op, 0\}$, the relation 
\begin{equation*}
d^{\diamond}_{\kk-1;j} \circ d^{\star}_{\kk;i}= d^{\star}_{\kk-1;i-1} \circ d^{\diamond}_{\kk;j}
\end{equation*}
is satisfied. Together with a careful sign juggling, this gives \eqref{E:Bicomplex}.
\end{proof}

\begin{coro}
\label{C:DistrHom}
Every $\ZZZZ$-linear combination $\der_{\kk}$ of~$\der^{\,\op}_{\kk}$ and $\der^{\,0}_{\kk}$ defines a chain complex structure on~$\CC_\kk(\SS)$, that is, we have $\der_{\kk-1} \circ \der_{\kk} = 0$.
\end{coro}

\begin{rema}
An alternative proof of the statement consists in interpreting~$\der^{\,0}$ as~$\der^{\,\OP0}$ for the trivial left-selfdistributive operation $\xx \OP0 \yy = \yy$, observing that operations~$\OP0$ and~$\op$ are mutually distributive---that is, $(\SS, \op, \OP0)$ is a multi-LD-system, as defined in~\cite{Dft} and~\cite{Larue} and rediscovered---under the name of multishelf---in~\cite{PrzHomo,PrzSi}, and applying the general multi-term distributive homology theory from~\cite{PrzHomo,PrzSi}.
\end{rema}

We then follow the standard terminology and notation.

\begin{defi}
\label{D:DistrHom}
Assume that $\SS$ is a set and $\op$ is a binary operation on~$\SS$ that obeys the left-distributive law.

\ITEM1 For every~$\kk \ge 1$, we put $\der^\Rack_{\kk}= \der^{\,\op}_{\kk}-\der^{\,0}_{\kk}$. Then the complex $(\CC_\kk(\SS),\der^\Rack_{\kk})$ is called the \emph{rack complex} of~$(\SS, \op)$, and its homology, denoted by~$H^\Rack_\kk(\SS),$ is called the \emph{rack homology} of~$(\SS, \op)$.

\ITEM2 Assume that $\GG$ is an abelian group, and consider the cochain complex $\CC^\kk(\SS;\GG)$ defined as the abelian group $\Hom_{\ZZZZ}(\CC_\kk(\SS), \GG)$ endowed with the differential $\der_\Rack^{\kk}$ induced by $\der^\Rack_{\kk}$. Then the functions in the image of~$\der_\Rack^{\kk-1}$ are called $\GG$-valued \emph{$\kk$-coboundaries} and their set is denoted by~$\BB_\Rack^\kk(\SS; \GG)$, whereas those in the kernel of~$\der_\Rack^\kk$ are called $\GG$-valued \emph{$\kk$-cocycles}, and their set is denoted by~$\ZZ_\Rack^\kk(\SS; \GG)$. The quotient $\ZZ_\Rack^\kk(\SS; \GG){/}\BB_\Rack^\kk(\SS; \GG)$ is called the $\GG$-valued \emph{rack cohomology} of~$(\SS, \op)$, and denoted by $\HH_\Rack^\kk(\SS; \GG)$.
\end{defi}

\begin{rema}
In the distributive world, the rack complex can be seen as the analogue of the Hochschild complex for associative algebras, whereas the complex $(\CC_\kk(\SS),\der^{\,\op}_{\kk})$, known as the 1-term distributive complex of~$(\SS, \op)$, as the analogue of the bar complex. This was pointed out in \cite{PrzHomo} and explained in the context of a unifying braided homology theory in \cite{Lebed1}.
\end{rema}

In what follows we mostly work with~$\GG = \ZZZZ$. However, all results extend to the case of an arbitrary abelian group with obvious modifications. 

For further reference, note the explicit values for~$\der_\Rack^\kk$ with $1 \le \kk \le 4$:
\begin{gather}
\label{E:Der1}
\der_\Rack^1 \phi\, (\xx) = 0,\\
\label{E:Der2}
\der_\Rack^2 \phi\, (\xx, \yy) = \phi(\xx \op \yy) - \phi(\yy),\\
\label{E:Der3}
\der_\Rack^3 \phi\, (\xx, \yy, \zz) = \phi(\xx \op \yy, \xx \op \zz) + \phi(\xx, \zz) - \phi(\xx, \yy \op \zz) - \phi(\yy, \zz),\\
\label{E:Der4}
\der_\Rack^4 \phi\, (\xx, \yy, \zz, \tt) = \phi(\xx \op \yy, \xx \op \zz, \xx \op \tt) + \phi(\xx, \yy, \zz \op \tt)\hspace{20mm}\\
\notag
\hspace{20mm}+ \phi(\xx, \zz, \tt) - \phi(\xx, \yy \op \zz, \yy \op \tt) - \phi(\yy, \zz, \tt) - \phi(\xx, \yy, \tt).
\end{gather}

Below we shall specifically consider $2$- and $3$-cocycles. Let us briefly recall why such cocycles are directly interesting for constructing topological invariants, more precisely for defining quantities that are invariant under Reidemeister~$\mathrm{III}$ moves. Let us begin with $2$-cocycles. Restarting with the colouring rule used in Figure~\ref{F:Colouring}, we can use a $2$-cocycle to attach an element of~$\GG$ to every crossing. The general result is then 

\begin{lemm}\cite{CJKLS}
\label{L:Use2Cocycle}
Assume that $(\SS, \op)$ is a left-selfdistributive structure, $\GG$ is an abelian group, and $\phi$ is a $\GG$-valued $2$-cocycle for~$\SS$. For $\DD$ a positive $\nn$-strand braid diagram and $\vec\aa$ in~$\SS^\nn$, define $\phih_\DD(\vec\aa) = \sum_\ii \phi(\bb_\ii, \cc_\ii)$ where $\bb_\ii, \cc_\ii$ are the input colours at the $\ii$th crossing of~$\DD$ when $\DD$ is coloured from~$\vec\aa$. Then $\phih_\DD$ is invariant Reidemeister~$\mathrm{III}$  moves.
\end{lemm} 

The verification is shown in Figure~\ref{F:Cocycle}.

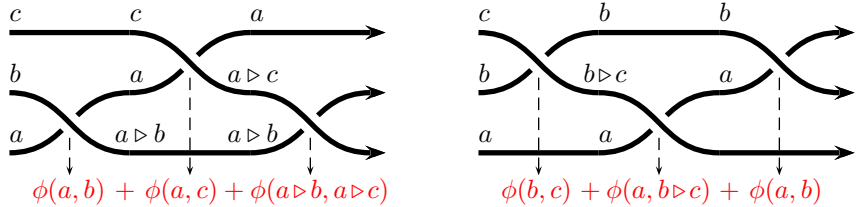
\begin{figure}[htb]
\begin{picture}(50,25)(0,-6)
\psbezier[linewidth=2pt,border=3pt](0,0)(8,0)(8,8)(16,8)
\psbezier[linewidth=2pt,border=3pt](0,8)(8,8)(8,0)(16,0)
\psline[linewidth=2pt,border=3pt](0,16)(16,16)
\psbezier[linewidth=2pt,border=3pt](16,8)(24,8)(24,16)(32,16)
\psbezier[linewidth=2pt,border=3pt](16,16)(24,16)(24,8)(32,8)
\psline[linewidth=2pt,border=3pt](16,0)(32,0)
\psbezier[linewidth=2pt,border=3pt](32,0)(40,0)(40,8)(48,8)
\psbezier[linewidth=2pt,border=3pt](32,8)(40,8)(40,0)(48,0)
\psline[linewidth=2pt,border=3pt](32,16)(48,16)
\psline[linewidth=2pt,border=3pt]{->}(48,0)(50,0)
\psline[linewidth=2pt,border=3pt]{->}(48,8)(50,8)
\psline[linewidth=2pt,border=3pt]{->}(48,16)(50,16)
\put(0,1.5){$\aa$}
\put(0,9.5){$\bb$}
\put(0,17.5){$\cc$}
\put(16,17.5){$\cc$}
\put(14,1.5){$\aa\op\bb$}
\put(16,9.5){$\aa$}
\put(29,9.5){$\aa\op\cc$}
\put(29,1.5){$\aa\op\bb$}
\put(32,17.5){$\aa$}
\psline[linewidth=0.5pt,linestyle=dashed]{->}(8,2)(8,-3)
\put(3,-6){\color{red}$\phi(\aa, \bb)$}
\put(14,-6){$\color{red}+$}
\put(28.5,-6){$\color{red}+$}
\psline[linewidth=0.5pt,linestyle=dashed]{->}(24,10)(24,-3)
\put(18,-6){\color{red}$\phi(\aa, \cc)$}
\psline[linewidth=0.5pt,linestyle=dashed]{->}(40,2)(40,-3)
\put(32,-6){\color{red}$\phi(\aa\ops\bb, \aa\ops\cc)$}
\end{picture}
\hspace{10mm}
\begin{picture}(50,25)(0,-6)
\psbezier[linewidth=2pt,border=3pt](0,8)(8,8)(8,16)(16,16)
\psbezier[linewidth=2pt,border=3pt](0,16)(8,16)(8,8)(16,8)
\psline[linewidth=2pt,border=3pt](0,0)(16,0)
\psbezier[linewidth=2pt,border=3pt](16,0)(24,0)(24,8)(32,8)
\psbezier[linewidth=2pt,border=3pt](16,8)(24,8)(24,0)(32,0)
\psline[linewidth=2pt,border=3pt](16,16)(32,16)
\psbezier[linewidth=2pt,border=3pt](32,8)(40,8)(40,16)(48,16)
\psbezier[linewidth=2pt,border=3pt](32,16)(40,16)(40,8)(48,8)
\psline[linewidth=2pt,border=3pt](32,0)(48,0)
\psline[linewidth=2pt,border=3pt]{->}(48,0)(50,0)
\psline[linewidth=2pt,border=3pt]{->}(48,8)(50,8)
\psline[linewidth=2pt,border=3pt]{->}(48,16)(50,16)
\put(0,1.5){$\aa$}
\put(0,9.5){$\bb$}
\put(0,17.5){$\cc$}
\put(16,17.5){$\bb$}
\put(16,1.5){$\aa$}
\put(14,9.5){$\bb\ops\cc$}
\put(32,9.5){$\aa$}
\put(32,17.5){$\bb$}
\psline[linewidth=0.5pt,linestyle=dashed]{->}(8,10)(8,-3)
\put(3,-6){\color{red}$\phi(\bb, \cc)$}
\put(13.5,-6){\color{red}$+$}
\psline[linewidth=0.5pt,linestyle=dashed]{->}(24,2)(24,-3)
\put(17,-6){\color{red}$\phi(\aa, \bb\ops\cc)$}
\put(32,-6){\color{red}$+$}
\psline[linewidth=0.5pt,linestyle=dashed]{->}(40,10)(40,-3)
\put(36,-6){\color{red}$\phi(\aa, \bb)$} 
\end{picture}
\caption{\sf\small Using a $2$-cocycle to construct a positive braid invariant: one associates with every braid diagram the sum of the values of the cocycle at the successive crossings labelled by means of the reference left-selfdistributive structure; the cocycle rule~\eqref{E:Cocycle} is exactly what is needed to guarantee invariance with respect to Reidemeister~$\mathrm{III}$ moves.}
\label{F:Cocycle}
\end{figure}

The principle for using $3$-cocycles is similar, at the expense of combining arc colouring with region colouring, as explained in Figure~\ref{F:ColouringShadow}---here we use the same auxiliary set~$\SS$ to colour arcs and regions, but more general rules involving what is called rack-sets or rack shadows are also possible, see \cite{FRSRackHom, FRSJames, Kam, CKS, ChNe}. 

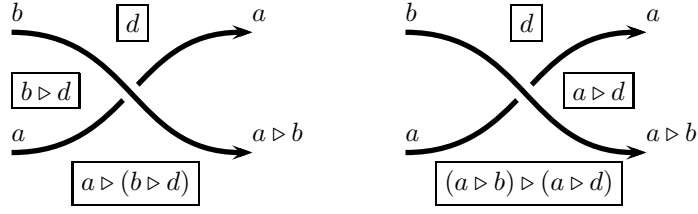
\begin{figure}[htb]
\begin{picture}(40,25)(0,-6)
\psbezier[linewidth=2pt,border=3pt]{->}(0,0)(16,0)(16,16)(32,16)
\psbezier[linewidth=2pt,border=3pt]{->}(0,16)(16,16)(16,0)(32,0)
\put(0,1.5){$\aa$}
\put(0,17.5){$\bb$}
\put(32,1.5){$\aa\op\bb$}
\put(32,17.5){$\aa$}
\put(14,16){$\boxed{\dd}$}
\put(0,7){$\boxed{\bb\op\dd}$}
\put(8,-5){$\boxed{\aa\op (\bb\op\dd)}$}
\end{picture}
\hspace{10mm}
\begin{picture}(40,25)(0,-6)
\psbezier[linewidth=2pt,border=3pt]{->}(0,0)(16,0)(16,16)(32,16)
\psbezier[linewidth=2pt,border=3pt]{->}(0,16)(16,16)(16,0)(32,0)
\put(0,1.5){$\aa$}
\put(0,17.5){$\bb$}
\put(32,1.5){$\aa\op\bb$}
\put(32,17.5){$\aa$}
\put(14,16){$\boxed{\dd}$}
\put(21,7){$\boxed{\aa\op\dd}$}
\put(4,-5){$\boxed{(\aa\op\bb)\op (\aa\op\dd)}$}
\end{picture}
\caption{\sf\small Region colouring: if the region above an $\aa$-coloured arc is coloured with~$\dd$, then the region below is coloured with~$\aa \ops \dd$; the coherence of this colouring rule is equivalent to the left-selfdistributivity law for the operation~$\op$.}
\label{F:ColouringShadow}
\end{figure}

The counterpart of Lemma~\ref{L:Use2Cocycle} is then 

\begin{lemm}\cite{FRSJames}
\label{L:Use3Cocycle}
In the settings of  Lemma~\ref{L:Use2Cocycle},  assume that $\phi$ is a $\GG$-valued $3$-cocycle for~$\SS$, and, together with arcs, colour the regions of the diagram~$\DD$, starting from colour~$\dd$ for the region on the top. Define $\phih_\DD(\vec\aa,\dd) = \sum_\ii \phi(\bb_\ii, \cc_\ii, \dd_\ii)$ where $\bb_\ii, \cc_\ii$ are the input colours and $\dd_\ii$ is the upper region colour at the $\ii$th crossing of~$\DD$. Then $\phih_\DD$ is invariant under Reidemeister~$\mathrm{III}$ moves.
\end{lemm} 

The proof is now contained in Figure~\ref{F:Cocycle3}.

\begin{figure}[htb]
\begin{picture}(55,30)(0,-6)
\psbezier[linewidth=2pt,border=3pt](0,0)(8,0)(8,8)(16,8)
\psbezier[linewidth=2pt,border=3pt](0,8)(8,8)(8,0)(16,0)
\psline[linewidth=2pt,border=3pt](0,16)(16,16)
\psbezier[linewidth=2pt,border=3pt](16,8)(24,8)(24,16)(32,16)
\psbezier[linewidth=2pt,border=3pt](16,16)(24,16)(24,8)(32,8)
\psline[linewidth=2pt,border=3pt](16,0)(32,0)
\psbezier[linewidth=2pt,border=3pt](32,0)(40,0)(40,8)(48,8)
\psbezier[linewidth=2pt,border=3pt](32,8)(40,8)(40,0)(48,0)
\psline[linewidth=2pt,border=3pt](32,16)(48,16)
\psline[linewidth=2pt,border=3pt]{->}(48,0)(50,0)
\psline[linewidth=2pt,border=3pt]{->}(48,8)(50,8)
\psline[linewidth=2pt,border=3pt]{->}(48,16)(50,16)
\put(0,1.5){$\aa$}
\put(0,9.5){$\bb$}
\put(0,17.5){$\cc$}
\put(16,9.5){$\aa$}
\put(29,9.5){$\aa\op\cc$}
\put(29,1.5){$\aa\op\bb$}
\put(36,17.5){$\aa$}
\put(22,19.5){$\boxed{\dd}$}
\put(4,10.5){$\boxed{\cc\op\dd}$}
\put(38,10.5){$\boxed{\aa\op\dd}$}
\psline[linewidth=0.5pt,linestyle=dashed]{->}(8,2)(8,-3)
\put(-4,-6){\color{red}$\phi(\aa, \bb,\cc\op\dd)$}
\psline[linewidth=0.5pt,linestyle=dashed]{->}(24,10)(24,-3)
\put(17,-6){\color{red}$\phi(\aa, \cc,\dd)$}
\psline[linewidth=0.5pt,linestyle=dashed]{->}(40,2)(40,-3)
\put(34,-6){\color{red}$\phi(\aa\ops\bb, \aa\ops\cc, \aa\op\dd)$}
\put(14.2,-6){$\color{red}+$}
\put(30.5,-6){$\color{red}+$}
\end{picture}
\hspace{10mm}
\begin{picture}(50,30)(0,-6)
\psbezier[linewidth=2pt,border=3pt](0,8)(8,8)(8,16)(16,16)
\psbezier[linewidth=2pt,border=3pt](0,16)(8,16)(8,8)(16,8)
\psline[linewidth=2pt,border=3pt](0,0)(16,0)
\psbezier[linewidth=2pt,border=3pt](16,0)(24,0)(24,8)(32,8)
\psbezier[linewidth=2pt,border=3pt](16,8)(24,8)(24,0)(32,0)
\psline[linewidth=2pt,border=3pt](16,16)(32,16)
\psbezier[linewidth=2pt,border=3pt](32,8)(40,8)(40,16)(48,16)
\psbezier[linewidth=2pt,border=3pt](32,16)(40,16)(40,8)(48,8)
\psline[linewidth=2pt,border=3pt](32,0)(48,0)
\psline[linewidth=2pt,border=3pt]{->}(48,0)(50,0)
\psline[linewidth=2pt,border=3pt]{->}(48,8)(50,8)
\psline[linewidth=2pt,border=3pt]{->}(48,16)(50,16)
\put(0,1.5){$\aa$}
\put(0,9.5){$\bb$}
\put(0,17.5){$\cc$}
\put(14,9.5){$\bb\ops\cc$}
\put(32,9.5){$\aa$}
\put(32,17.5){$\bb$}
\put(22,19.5){$\boxed{\dd}$}
\put(21,10.5){$\boxed{\bb\op\dd}$}
\psline[linewidth=0.5pt,linestyle=dashed]{->}(8,10)(8,-3)
\put(-3,-6){\color{red}$\phi(\bb, \cc,\dd)$}
\put(10,-6){$\color{red}+$}
\psline[linewidth=0.5pt,linestyle=dashed]{->}(24,2)(24,-3)
\put(13,-6){\color{red}$\phi(\aa, \bb\ops\cc,\bb\op\dd)$}
\put(35.6,-6){$\color{red}+$}
\psline[linewidth=0.5pt,linestyle=dashed]{->}(40,10)(40,-3)
\put(39,-6){\color{red}$\phi(\aa, \bb,\dd)$} 
\end{picture}
\caption{\sf\small Using a $3$-cocycle to construct a positive braid invariant: the construction is similar to that using a $2$-cocycle, but this time the upper region colour is taken into consideration at each crossing; the cocycle rule~\eqref{E:Cocycle4} is exactly what is needed to guarantee invariance with respect to Reidemeister~$\mathrm{III}$ moves; only the colours of the regions relevant to our construction are indicated in the figure.}
\label{F:Cocycle3}
\end{figure}
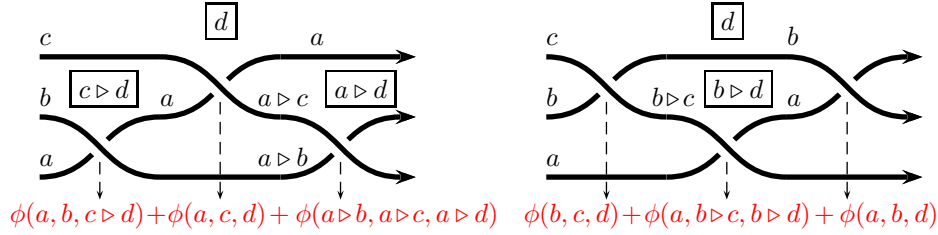

In practice, the same $\SS$ often admits considerably more $3$-cocycles than $2$-cocycles, so the former allow to extract more invariants out of the same reference structure. A construction of knotted surface invariants in $\RRRR^4$ out of $3$-cocycles is another motivation for hunting for $3$-cocycles, see \cite{Kam,CarterSurvey} for details.

\section{Two-cocycles for Laver tables}
\label{S:TwoPhi}

 From now on, our aim is to investigate the rack cohomology of Laver tables and, more precisely, to determine the associated cocycles. As a warm-up, we consider the (simple) case of degree~$1$ rack cohomology.
 
\begin{prop}
\label{P:H1} 
For every~$\nn \ge 0$, the first rack cohomology group $\HH_\Rack^1(\AA_\nn; \ZZZZ)$ is a free $1$-dimensional $\ZZZZ$-module generated by the equivalence class of the constant map with value~$1$. \end{prop}
 
\begin{proof}
According to~\eqref{E:Der2}, a $\ZZZZ$-valued $1$-cocycle on~$\AA_\nn$ is a map $\phi : \AA_\nn \to \ZZZZ$ that satisfies $\phi(\xx \op \yy) = \phi(\yy)$ for all~$\xx, \yy$, hence in particular $\phi(\xx \op 1) = \phi(1)$ for every~$\xx$. As $\xx \op 1$ ranges from~$1$ to~$2^\nn$ when $\xx$ varies, we deduce that $\phi$ must be constant. On the other hand, any constant map $\phi$ satisfies $\phi(\xx \op \yy) = \phi(\yy)$  and is thus a cocycle. Hence $\ZZ_\Rack^1(\AA_\nn; \ZZZZ)$ is generated by the constant map with value~$1$, and it is isomorphic to~$\ZZZZ$. 
 Further, due to~\eqref{E:Der1}, one has $\BB_\Rack^1(\AA_\nn; \ZZZZ) \cong 0$, so $\HH_\Rack^1(\AA_\nn; \ZZZZ) \cong \ZZ_\Rack^1(\AA_\nn; \ZZZZ) \cong \ZZZZ$.
\end{proof}

In fact, the argument above can be used for showing that the first rack homology and cohomology groups of any  selfdistributive structure that is \emph{monogenerated} (that is, generated by a single element) are $1$-dimensional.

 For the rest of this section, we consider degree~$2$ rack cohomology and, more specifically,  $\ZZZZ$-valued  $2$-cocycles. By definition, a $\ZZZZ$-valued $2$-cocycle for~$\AA_\nn$ is a $\ZZZZ$-linear map from~$\ZZZZ\AA_\nn \times \AA_\nn$ to~$\ZZZZ$, and it is fully determined by its values on~$\AA_\nn \times \AA_\nn$. Thus, we are looking for maps $\phi:\{1 \wdots 2^\nn\}\times\{1 \wdots 2^\nn\} \rightarrow \ZZZZ$ satisfying
\begin{equation}
\label{E:Cocycle}
\phi(\xx \op \yy, \xx \op \zz) + \phi(\xx, \zz) = \phi(\xx, \yy \op \zz) + \phi(\yy, \zz).
\end{equation}
 As in the case of the operation~$\OP\nn$ itself, it is natural to represent a $2$-cocycle~$\phi$ for~$\AA_\nn$ as a square table based on~$\{1 \wdots 2^\nn\}$, with the $(\xx,\yy)$-entry containing the integer~$\phi(\xx,\yy)$. Therefore, we can speak of the rows and columns of a $2$-cocycle. Further, we shall omit the subscripts~$\Rack$ and the coefficients~$\ZZZZ$ for brevity, writing for instance $\ZZ^2(\AA_\nn)$ instead of $\ZZ_\Rack^2(\AA_\nn;\ZZZZ)$.

 We  start with an easy obervation: since every $2$-coboundary is a $2$-cocycle, we can describe a number of elements in~$\ZZ^2(\AA_\nn)$ using the definition~\eqref{E:Der2} of~$\der^2$.
 
\begin{lemm}
\label{L:Delta}
For $\pp, \qq$ in~$\AA_\nn$, let $\delta_{\pp, \qq}$ be~$1$ for~$\pp = \qq$ and be~$0$ for~$\pp \neq \qq$. For $1 \le \qq \le 2^\nn$, define $\phi_{\qq, \nn}$ by 
\begin{equation}
\label{E:phi}
\phi_{\qq, \nn}(\xx, \yy) = \delta_{\yy, \qq} - \delta_{\xx \OP\nn \yy, \qq}.
\end{equation}
Then, for every~$\qq$, the map~$\phi_{\qq, \nn}$ defines a $2$-coboundary on~$\AA_\nn$.
\end{lemm}

\begin{proof}
Let $\theta$ be the $1$-cochain on~$\AA_\nn$ defined by~$\theta(\xx) = -\delta_{\xx, \qq}$. Then \eqref{E:Der2} shows that $\phi_{\qq, \nn}$ is equal to~$\der^2\theta$. Hence $\phi_{\qq, \nn}$ belongs to~$\BB^2(\AA_\nn)$. 
\end{proof}

The table of the coboundary~$\phi_{\qq, \nn}$  is easily described in terms of the occurrences of the value~$\qq$ in the table of~$\AA_\nn$. Namely, it consists of zeroes everywhere, except for ones in the $\qq$th column at the positions where the value is not~$\qq$ in~$\AA_\nn$, and minus ones in the other columns at the positions where the value is~$\qq$, see Table~\ref{T:ExamplesPhi}.

\begin{table}[htb]
\small\small
\begin{tabular}{c|c}
\VR(0,1.5)\hidewidth{$\phi_{1,3}$}&\1\2\3\4\5\6\7\8\\
\hline
\VR(3,0)$1$&\1\0\0\0\0\0\0\0\\
$2$&\1\0\0\0\0\0\0\0\\
$3$&\1\0\0\0\0\0\0\0\\
$4$&\1\0\0\0\0\0\0\0\\
$5$&\1\0\0\0\0\0\0\0\\
$6$&\1\0\0\0\0\0\0\0\\
$7$&\1\0\0\0\0\0\0\0\\
$8$&\0\0\0\0\0\0\0\0
\end{tabular}
\qquad
\begin{tabular}{c|c}
\VR(0,1.5)\hidewidth{$\phi_{4,3}$}&\1\2\3\4\5\6\7\8\\
\hline
\VR(3,0)$1$&\0\mo\0\1\0\mo\0\0\\
$2$&\0\mo\0\1\0\mo\0\0\\
$3$&\mo\0\mo\1\mo\0\mo\0\\
$4$&\0\0\0\1\0\0\0\0\\
$5$&\0\0\0\1\0\0\0\0\\
$6$&\0\0\0\1\0\0\0\0\\
$7$&\0\0\0\1\0\0\0\0\\
$8$&\0\0\0\0\0\0\0\0
\end{tabular}
\qquad
\begin{tabular}{c|c}
\VR(0,1.5)\hidewidth{$\phi_{7,3}$}&\1\2\3\4\5\6\7\8\\
\hline
\VR(3,0)$1$&\0\0\0\0\0\0\1\0\\
$2$&\0\0\mo\0\0\0\0\0\\
$3$&\0\0\0\0\0\0\1\0\\
$4$&\0\0\mo\0\0\0\0\0\\
$5$&\0\0\0\0\0\0\1\0\\
$6$&\mo\0\mo\0\mo\0\0\0\\
$7$&\0\0\0\0\0\0\1\0\\
$8$&\0\0\0\0\0\0\0\0
\end{tabular}

\VR(2,0)
\caption{\sf\small Three $2$-cocycles on~$\AA_3$, here $\phi_{1,3}$, $\phi_{4,3}$, and $\phi_{7,3}$: in each case, the values are $0$ (replaced with a dot for readability), $1$, or $-1$, depending on the column and the corresponding value in the considered Laver table.}
\label{T:ExamplesPhi}
\end{table}

We shall now establish that the first $2^\nn-1$ coboundaries of Lemma~\ref{L:Delta} make a basis of~$\BB^2(\AA_\nn)$ and that, when completed with a constant cocycle, they make a basis of~$\ZZ^2(\AA_\nn)$.

\begin{prop}
\label{P:Basis1}
 For every~$\nn$, the group~$\ZZ^2(\AA_\nn)$ is a free abelian group of rank~$2^\nn$, with a basis consisting of the coboundaries~$\phi_{\pp, \nn}$ with $\pp < 2^\nn$ and the constant $2$-cocycle~$\const_\nn$ with value~$1$. Every $2$-cocycle~$\phi$ for~$\AA_\nn$ admits the decomposition 
\begin{equation}
\label{E:Decomp}
\phi = \VR(5,3)\smash{\sum_{\qq=1}^{2^\nn{-}1}} (\phi(\ppn, \qq) -\phi(\ppn, 2^\nn))\phi_{\qq, \nn} + \phi(\ppn, 2^\nn) \const_\nn.
\end{equation}
Moreover, the subfamily $\{\phi_{1, \nn} \wdots \phi_{2^\nn{-}1, \nn}\}$ is a basis of $\BB^2(\AA_\nn)$, and $\HH^2(\AA_\nn) \cong \ZZZZ$ holds.
\end{prop}

The proof, which is not difficult, relies on two auxiliary results exploiting the (very) specific properties of Laver tables. 
 
\begin{lemm}
\label{L:LastColumn2}
Assume that $\phi$ is a $2$-cocycle for~$\AA_\nn$. Then, for~$1 \le \xx \le 2^\nn$, we have $\phi(\xx, 2^\nn)= \phi(\ppn, 2^\nn)$.
\end{lemm}

\begin{proof}
Write $\pn$ for~$\ppn$  and $\op$ for $\OP\nn$. Applying~\eqref{E:Cocycle} to $(\pn, \xx, 2^\nn)$  yields
\begin{equation}
\label{E:Cocycle13'}
\phi(\pn \op \xx, \pn \op 2^\nn) + \phi(\pn, 2^\nn) = \phi(\pn, \xx \op 2^\nn) + \phi(\xx, 2^\nn).
\end{equation}
According to~\eqref{E:LastColumn} and~\eqref{E:LastRows}, for every~$\xx$ we have $\pn\op\xx = 2^\nn$ and $\xx \op 2^\nn = 2^\nn$, so \eqref{E:Cocycle13'} reduces to
$\phi(2^\nn, 2^\nn) = \phi(\xx,2^\nn)$. Thus $\phi(\xx,2^\nn)$ does not depend on~$\xx$.  
\end{proof}

In other words, the last column of every $2$-cocycle for~$\AA_\nn$ is constant. The second auxiliary result states that a $2$-cocycle with a trivial penultimate row must be trivial.

\begin{lemm}
\label{L:PenultRow2}
Assume that $\phi$ is a $2$-cocycle for~$\AA_\nn$ and $\phi(\ppn, \yy) = 0$ holds for every~$\yy$. Then $\phi$ is the zero cocycle.
\end{lemm}

\begin{proof}
 Write $\pn$ for~$\ppn$ again. Applying~\eqref{E:Cocycle} to~$(\pn, \xx, \yy)$ yields
\begin{equation*}
\phi(\pn \op \xx, \pn \op \yy) + \phi(\pn, \yy) = \phi(\pn, \xx \op \yy) + \phi(\xx, \yy),
\end{equation*}
hence $\phi(2^\nn, 2^\nn) = \phi(\xx, \yy)$, owing to~\eqref{E:LastRows} and the assumption on~$\phi$. Lemma~\ref{L:LastColumn2} gives $\phi(2^\nn, 2^\nn) = \phi(\pn,2^\nn)$, which is zero, again by  the assumption on~$\phi$. One concludes that $\phi$ is zero on~$\AA_\nn \times \AA_\nn$.
\end{proof}

 We can now complete the argument. 

\begin{proof}[Proof of Proposition~\ref{P:Basis1}]
 As usual write $\pn$ for~$\ppn$. From the definition~\eqref{E:phi} and from the values~\eqref{E:LastRows} in the penultimate row of a Laver table, we find for $\qq <2^\nn$
\begin{equation*}
\phi_{\qq, \nn}(\pn, \yy) =  \delta_{\yy, \qq} -  \delta_{\pn \OP\nn \yy, \qq} =  \delta_{\yy, \qq} -  \delta_{2^\nn, \qq}  =  \delta_{\yy, \qq},
\end{equation*}
that is, $\phi_{\qq, \nn}$ is a $2$-cocycle whose penultimate row contains~$1$ in the $\qq$th column and zeroes everywhere else. This cocycle system can be completed with the cocycle $\const'_\nn$ defined by 
\begin{equation}
\label{E:cbar}
\const'_\nn=\cc_\nn - {\sum\nolimits_{\qq=1}^{\pn}}\phi_{\qq, \nn},
\end{equation}
whose penultimate row contains~$1$ in the last column and zeroes everywhere else. These calculations show that an arbitrary $2$-cocycle $\phi$ for~$\AA_\nn$ coincides on the penultimate row with the cocycle $\VR(4,2)\smash{\sum_{\qq=1}^{\pn}} \phi(\pn, \qq)\phi_{\qq, \nn} + \phi(\pn, 2^\nn) \const'_\nn,$ which is precisely the right side of~\eqref{E:Decomp}. By Lemma~\ref{L:PenultRow2}, these cocycles must coincide and, therefore, Relation~\eqref{E:Decomp} holds. Consequently, the family $\{\phi_{1, \nn} \wdots \phi_{2^\nn-1, \nn}, \const'_\nn\}$, and hence $\{\phi_{1, \nn} \wdots \phi_{2^\nn-1, \nn}, \const_\nn\}$, gene\-rates $\ZZ^2(\AA_\nn)$. Let us now show that this family is free. Suppose that $\VR(4,2)\sum_{\qq=1}^{\pn} \lambda_\qq \phi_{\qq, \nn} + \lambda_{2^\nn} \const_\nn$ is a zero $2$-cocycle for some integer coefficients~$\lambda_i$. Evaluating at~$(\pn,2^\nn)$ and using the calculations above, one  obtains $\lambda_{2^\nn}=0$. Further, evaluating at~$(\pn,\qq)$ with~$\qq <2^\nn$, one  obtains $\lambda_{\qq}=0$ for all~$\qq$. Thus $\ZZ^2(\AA_\nn)$ is a free $\ZZZZ$-module based on $\{\phi_{1, \nn} \wdots \phi_{2^\nn-1, \nn}, \const_\nn\}$.

It remains to prove the assertion for~$\BB^2(\AA_\nn)$. According to Lemma~\ref{L:Delta}, the maps $\phi_{1, \nn} \wdots \phi_{2^\nn-1, \nn}$ lie in~$\BB^2(\AA_\nn)$, and, as argued above, they form a free family. Let us show that they generate the whole~$\BB^2(\AA_\nn)$. Due to the definition~\eqref{E:Der2} and the property~\eqref{E:LastColumn} of Laver tables, every $2$-coboundary~$\der^2\theta$ on~$\AA_\nn$ satisfies 
$$\der^2\theta(\pn,2^\nn) = \theta(\pn \op 2^\nn) - \theta(2^\nn) = \theta(2^\nn) - \theta(2^\nn) = 0,$$
hence the last term in the decomposition~\eqref{E:Decomp} for~$\der^2\theta$ is zero, and it belongs to the submodule of~$\ZZ^2(\AA_\nn)$ generated by $\{\phi_{1, \nn} \wdots \phi_{2^\nn-1, \nn}\}$.\end{proof}

Once a basis of $\ZZ^2(\AA_\nn)$ is known, it becomes extremely easy to extract common properties of its elements and thus establish general properties of the cocycles. Here are some examples.

\begin{prop}
Assume that $\phi$ is a $2$-cocycle for~$\AA_\nn$. Put $\uu = \phi(1, 2^{\nn-1})$ and $\vv = \phi(2^\nn, 2^\nn)$.

\ITEM1 For all~$\qq \le 2^\nn$, we have $\phi(2^\nn,\qq) = \vv$, that is, the last row of~$\phi$ is constant.

\ITEM2 For all~$\pp < 2^\nn$, we have $\phi(\pp, 2^{\nn-1}) = \uu$, that is, the column of~$2^{\nn-1}$ is constant, except possibly its last element. For every $\qq > 2^{\nn-1}$, we have $\phi(2^{\nn-1}, \qq) =\nobreak \vv$, that is, the second half of the row of $2^{\nn-1}$ is constant.

\ITEM3 For every~$\qq < 2^{\nn-1}$, we have $\phi(\pp, \qq) = \phi(\pp, \qq + 2^{\nn-1})$ for every~$\pp$, that is, the columns of $\qq$ and $\qq + 2^{\nn-1}$ coincide, if and only if $\phi(2^{\nn-1}, \qq) = \vv$ holds.
\end{prop}

\begin{proof}
All properties in~\ITEM1 and~\ITEM2 are satisfied by each of the cocycles~$\phi_{\ii, \nn}$ and $\const_\nn$, and they are preserved under linear combination. By Proposition~\ref{P:Basis1}, they are therefore satisfied by every cocycle.  The situation with~\ITEM3 is similar, but the preservation under linearity requires (slightly) more care. So we shall instead make a direct verification---similar verifications are of course possible in the cases of~\ITEM1 and~\ITEM2.  

Making $\yy = 2^{\nn-1}$ and $\zz = \qq$ in~\eqref{E:Cocycle} yields
\begin{equation}
\label{E:Cocycle17}
\phi(\xx, \qq) + \phi(\xx\op2^{\nn-1}, \xx\op\qq) = \phi(2^{\nn-1}, \qq) + \phi(\xx, 2^{\nn-1}\op\qq).
\end{equation}
Assume $\xx < 2^\nn$ and $\qq < 2^{\nn-1}$. Then we have $\xx\op2^{\nn-1} = 2^\nn$ and $2^{\nn-1}\op\qq = \qq + 2^{\nn-1}$, so \eqref{E:Cocycle17} reduces to
\begin{equation*}
\phi(\xx, \qq) + \phi(2^\nn, \xx\op\qq) = \phi(2^{\nn-1}, \qq) + \phi(\xx, \qq + 2^{\nn-1}).
\end{equation*}
 It follows that the  equality $\phi(\xx, \qq) = \phi(\xx, \qq + 2^{\nn-1})$  is equivalent to $\phi(2^{\nn-1}, \qq)  = \phi(2^\nn, \xx\op\qq)$, that is, according to~\ITEM1, to $\phi(2^{\nn-1}, \qq)  = \vv.$
In the remaining case $\xx = 2^\nn$, note that, again by~\ITEM1, both $\phi(\xx, \qq)$ and $\phi(\xx, \qq + 2^{\nn-1})$ are equal to $\vv$ and thus coincide. We conclude that  the  equality $\phi(2^{\nn-1}, \qq)  = \vv$  holds if and only if  $\phi(\xx, \qq) = \phi(\xx, \qq + 2^{\nn-1})$ holds for every~$\xx$.
\end{proof}

\section{A partial ordering on~$\AA_\nn$}
\label{S:Poset}

In order to subsequently continue our investigation of $2$-cocycles for Laver tables, we shall now describe a certain partial order connected with right-division in the structure~$\AA_\nn$. Although this order seems to have never been mentioned explicitly, most essential properties of right-division in~$\AA_\nn$  were  established (with a different phrasing) in a series of papers by~A.\,Dr\'apal, who intensively investigated Laver tables in the 1990's~\cite{DraHom, DraPer, DraSem, DraAlg, DraGro, DraLow}. Below we shall give new, self-contained proofs of these properties which hopefully are more accessible than the original ones, spread in several sources; at the same time, this will provide typical examples of the delicate inductive arguments relevant in the investigation of Laver tables. From these properties, we shall deduce the existence of the above-mentioned partial order, finishing the section with its brief study.

\begin{defi}
For $1 \le \qq \le 2^\nn$, we put $\Col_\nn(\qq) = \{\pp \OP\nn \qq \mid \pp = 1 \wdots 2^\nn\}$, and we write  $\qq \Before\nn \rr$ if $\Col_\nn(\qq)$ contains~$\rr$.
\end{defi} 

So $\Col_\nn(\qq)$ is the family of all values that appear in the $\qq$th column of the Laver table~$\AA_\nn$. By definition, the relation  $\qq \Before\nn \rr$  holds if and only if  $\rr$ appears in the $\qq$th column of~$\AA_\nn$, hence if and only if there exists~$\pp$ satisfying $\pp \OP\nn \qq = \rr$, that is, if $\rr$ is what can be called a \emph{left-multiple}  of~$\qq$ (and, equivalently, $\qq$ is a \emph{right-divisor} of~$\rr$)  in~$\AA_\nn$. We shall establish below

\begin{prop}
\label{P:Poset}
For every~$\nn$, the relation~$\Before\nn$ is a partial ordering on~$\{1 \wdots 2^\nn \}$.  For every~$\pp$, we have $1 \Before\nn \pp \Before\nn 2^\nn$, that is, $1$ is initial and $2^\nn$ is final for~$\Before\nn$.
\end{prop}

 The Hasse diagrams of the partial orders $\before_2$, $\before_3$, and $\before_4$ are displayed  in Figure~\ref{F:Poset}.  As can be seen there,  $\Before\nn$ is not a linear order for~$\nn \ge 3$; it is not even a lattice order for~$\nn \ge 5$:  one can check that  $18$ and~$19$ admit no least upper bound with respect to~$\before_5$, as one has $18 \before_5 12$, $19 \before_5 12$, $18 \before_5 14$, and $19 \before_5 14$, but $18$ and $19$ admit no common upper bound that is a common lower bound of~$12$ and~$14$.

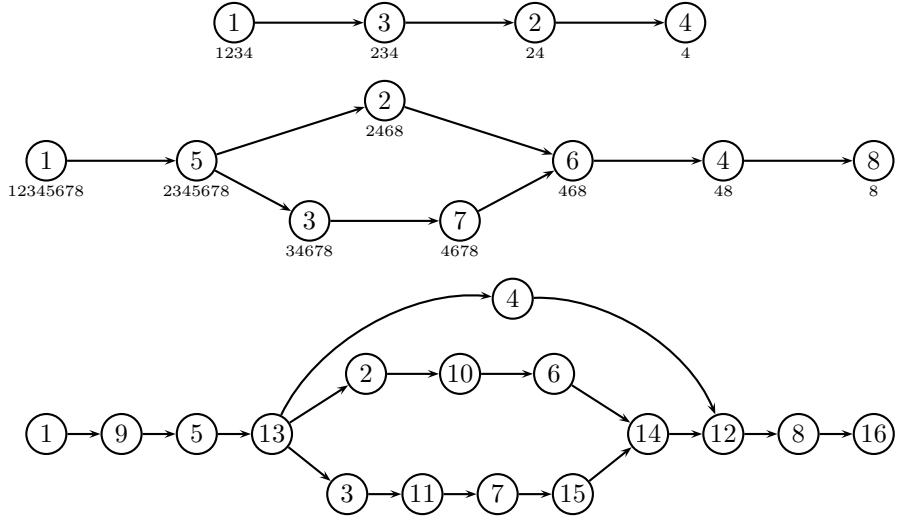
\begin{figure}[htb]
\begin{picture}(60,12)(0,0)
\psset{fillstyle=none}
\cnodeput(0,8){N1}{1}
\rput[c](0,4){\tiny$1234$}
\cnodeput(20,8){N2}{3}
\rput[c](20,4){\tiny$234$}
\cnodeput(40,8){N3}{2}
\rput[c](40,4){\tiny$24$}
\cnodeput(60,8){N4}{4}
\rput[c](60,4){\tiny$4$}
\ncline{->}{N1}{N2}
\ncline{->}{N2}{N3}
\ncline{->}{N3}{N4}
\end{picture}

\begin{picture}(110,22)(0,-4)
\psset{fillstyle=none}
\cnodeput(0,8){N1}{1}
\rput[c](0,4){\tiny$12345678$}
\cnodeput(20,8){N5}{5}
\rput[c](20,4){\tiny$2345678$}
\cnodeput(35,0){N3}{3}
\rput[c](35,-4){\tiny$34678$}
\cnodeput(45,16){N2}{2}
\rput[c](45,12){\tiny$2468$}
\cnodeput(55,0){N7}{7}
\rput[c](55,-4){\tiny$4678$}
\cnodeput(70,8){N6}{6}
\rput[c](70,4){\tiny$468$}
\cnodeput(90,8){N4}{4}
\rput[c](90,4){\tiny$48$}
\cnodeput(110,8){N8}{8}
\rput[c](110,4){\tiny$8$}
\ncline{->}{N1}{N5}
\ncline{->}{N5}{N3}
\ncline{->}{N5}{N2}
\ncline{->}{N3}{N7}
\ncline{->}{N7}{N6}
\ncline{->}{N2}{N6}
\ncline{->}{N6}{N4}
\ncline{->}{N4}{N8}
\end{picture}

\begin{picture}(110,36)(0,-4)
\psset{fillstyle=none}
\cnode(0,8){2.8}{N1}
\rput[c](0,8){1}
\cnode(10,8){2.8}{N9}
\rput[c](10,8){9}
\cnode(20,8){2.8}{N5}
\rput[c](20,8){5}
\cnode(30,8){2.8}{N13}
\rput[c](30,8){13}
\cnode(40,0){2.8}{N3}
\rput[c](40,0){3}
\cnode(50,0){2.8}{N11}
\rput[c](50,0){11}
\cnode(60,0){2.8}{N7}
\rput[c](60,0){7}
\cnode(70,0){2.8}{N15}
\rput[c](70,0){15}
\cnode(80,8){2.8}{N14}
\rput[c](80,8){14}
\cnode(90,8){2.8}{N12}
\rput[c](90,8){12}
\cnode(100,8){2.8}{N8}
\rput[c](100,8){8}
\cnode(110,8){2.8}{N16}
\rput[c](110,8){16}
\cnode(42.5,16){2.8}{N2}
\rput[c](42.5,16){2}
\cnode(55,16){2.8}{N10}
\rput[c](55,16){10}
\cnode(67.5,16){2.8}{N6}
\rput[c](67.5,16){6}
\cnode(62,26){2.8}{N4}
\rput[c](62,26){4}
\ncline{->}{N1}{N9}
\ncline{->}{N9}{N5}
\ncline{->}{N5}{N13}
\ncline{->}{N13}{N3}
\ncline{->}{N13}{N2}
\ncarc[arcangle=35]{->}{N13}{N4}
\ncarc[arcangle=35]{->}{N4}{N12}
\ncline{->}{N3}{N11}
\ncline{->}{N11}{N7}
\ncline{->}{N7}{N15}
\ncline{->}{N15}{N14}
\ncline{->}{N2}{N10}
\ncline{->}{N10}{N6}
\ncline{->}{N6}{N14}
\ncline{->}{N14}{N12}
\ncline{->}{N12}{N8}
\ncline{->}{N8}{N16}
\end{picture}
\caption{\sf\small The (Hasse diagrams of the) partial orders~$\before_2$ on~$\{1 \wdots 4\}$, $\before_3$ on~$\{1 \wdots 8\}$, and $\before_4$ on $\{1 \wdots 16\}$; in the top two diagrams, each node is accompanied with an enumeration of the corresponding column;  the relation~$\before_2$ is a linear order,  whereas $\before_3$ and $\before_4$ are  not linear orders since $2$ and~$3$ are not comparable, but they are lattice orders since any two elements admit a least upper bound (and a greatest lower bound).}
\label{F:Poset}
\end{figure}

Proposition~\ref{P:Poset} is somehow surprising: in an associative context, the counterpart of the $\before$~relation is trivially transitive since $\pp' \op (\pp \op \qq)$ is equal to $(\pp' \op \pp) \op \qq$, so a left-multiple of a left-multiple is still a left-multiple. But this need not be the case in a selfdistributive context, and the relation~$\before$, which makes sense in any algebraic context, need not be an ordering. By the way, the symmetric version of~$\Before\nn$ in~$\AA_\nn$ is not transitive for  $\nn>0$: writing  $\qq \beforesym_{\!\nn} \rr$  for $\exists\pp(\qq \OP\nn \pp = \rr)$, we have for instance $1 \beforesym_{\!\nn} 2^\nn \beforesym_{\!\nn} 1$, but not $1 \beforesym_{\!\nn} 1$---however, it can be mentioned that, in the case of a free left-selfdistributive structure on one generator, the transitive closure of the relation~$\beforesym$ turns out to be a linear ordering, a deep nontrivial result, see~\cite{Dgd}.

Before proving Proposition~\ref{P:Poset}, we begin with auxiliary results (of independent interest). The first task is to understand when the value~$\qq$ may appear in the column of~$\qq$. We first establish a positive result, which essentially corresponds to Proposition~1.4 in~\cite{DraSem}.

\begin{lemm}
\label{L:Valuation}
Assume $1 \le \pp \le 2^\nn$ and $2^\dd \div \pp$. Then 
\begin{equation}
\label{E:Valuation}
\pp \OP\nn (2^\nn - \qq) = 2^\nn - \qq
\end{equation}
holds for $0 \le \qq < 2^\dd$. 
\end{lemm}

In other words, if $2^\dd$ divides~$\pp$, then the last $2^\dd$ values in the row of~$\pp$ in the table of~$\AA_\nn$ are the consecutive numbers $2^\nn {-} 2^\dd {+} 1, 2^\nn {-} 2^\dd {+} 2 \wdots 2^\nn$. For instance, the last $8$ values in the rows of $8$, $16$, $24$, and~$32$ in the table of~$\AA_5$ must be $25, 26 \wdots 32$.

\begin{proof}
We establish the result using induction on~$\nn$. For $\nn = 0$, the result is trivial and, for $\nn = 1$, it says that the last value of every row in~$\AA_1$ is~$2$, whereas the last two values in the row of~$2$ are~$1$ and~$2$, which is true. 

From now on, we assume $\nn \ge 2$ and use induction on~$\dd$ increasing from~$0$ to~$\nn$. For $\dd = 0$, the statement says that the last value in the row of~$\pp$ is~$2^\nn$, which is indeed true. Assume now $1 \le \dd < \nn$ and $2^\dd \div \pp$. We shall establish~\eqref{E:Valuation} using induction on~$\qq$ increasing from~$0$ to~$2^\dd - 1$. For $0 \le \qq < 2^{\dd-1}$, \eqref{E:Valuation} follows from the induction hypothesis on~$\dd$. Assume now $2^{\dd-1} \le \qq < 2^\dd$. Put $\bar\pp = \MOD\pp{2^{\nn-1}}$. By assumption, we have $\dd \le \nn-1$, so $2^\dd$ also divides~$\bar\pp$, and $\qq < 2^{\nn - 1}$. Hence the induction hypothesis on~$\nn$ implies 
$$\bar\pp \ \op_{\nn-1}\  (2^{\nn-1} - \qq) = 2^{\nn-1} - \qq.$$ 
Lifting this equality from~$\AA_{\nn-1}$ to~$\AA_\nn$, we deduce that we have either~\eqref{E:Valuation}, or
\begin{equation}
\label{E:Valuation1}
\pp \OP\nn (2^\nn - \qq) = 2^{\nn-1} - \qq,
\end{equation}
and our problem is to exclude~\eqref{E:Valuation1}. So assume for a contradiction that \eqref{E:Valuation1} is true, and consider the next value in the row of~$\pp$, namely $\pp \OP\nn (2^\nn {-} \qq {+}1)$. 
On the one hand, by induction hypothesis, \eqref{E:Valuation} is true for $\qq - 1$, that is, we have
\begin{equation}
\label{E:Valuation2}
\pp \ \OP\nn\  (2^\nn - \qq + 1) = 2^\nn  - \qq + 1.
\end{equation}
On the other hand, \eqref{E:Laver2} and \eqref{E:Valuation1} give 
$$\pp \OP\nn (2^\nn {-} \qq {+}1) = (\pp \OP\nn (2^\nn {-} \qq)) \OP\nn (\pp \OP\nn 1) = (2^{\nn-1} {-} \qq) \OP\nn (\pp+1),$$
whence, merging with~\eqref{E:Valuation2},
\begin{equation}
\label{E:Valuation3}
(2^{\nn-1} - \qq) \ \OP\nn\  (\pp+1) = 2^\nn - \qq + 1.
\end{equation}
We shall see that \eqref{E:Valuation3} is impossible because it implies contradictory constraints for $\per_\nn(2^{\nn-1}-\qq)$, the period of $2^{\nn-1} - \qq$ in~$\AA_\nn$. 

Indeed, let us analyze the row of $2^{\nn-1} {-} \qq$ in~$\AA_\nn$. By~\eqref{E:Valuation3}, the value at~$\pp + 1$ is $2^\nn - \qq + 1$. As values in a row increase until~$2^\nn$ is reached, the smallest positive~$\rr$ for which we have
\begin{equation}
\label{E:Valuation4}
(2^{\nn-1} - \qq) \ \OP\nn\ (\pp + \rr) = 2^\nn.
\end{equation}
must satisfy $\rr \le \qq$. 
Then the row of $2^{\nn-1} {-} \qq$ contains at least $\rr+1$ values, namely the $\rr$~values at $\pp+1 \wdots \pp + \rr$, which are above $2^\nn {-} \qq$, hence larger than~$2^{\nn-1}$, plus the value at~$1$, which is $2^{\nn-1} {-} \qq {+} 1$, smaller than or equal to~$2^{\nn-1}$ and therefore different from the previous $\rr$~values. Hence, we must have 
\begin{equation}
\label{E:Valuation5}
\per_\nn(2^{\nn-1} - \qq) \ge \rr + 1. 
\end{equation}
On the other hand, \eqref{E:Valuation4} implies $\per_\nn(2^{\nn-1} - \qq) \div (\pp + \rr)$, whence $\per_\nn(2^{\nn-1} - \qq) \div \rr$ since we have $2^\dd \div \pp$ and $\rr \le \qq < 2^\dd$, so that  every  power of~$2$ dividing~$\pp + \rr$ also divides~$\rr$.  Hence, in particular, so does  $\per_\nn(2^{\nn-1} - \qq)$. We deduce $\per_\nn(2^{\nn-1} - \qq) \le \rr$, contradicting~\eqref{E:Valuation5}. Thus  \eqref{E:Valuation3} and, therefore, \eqref{E:Valuation1}, are impossible, and \eqref{E:Valuation} must be true.

Finally, for $\dd = \nn$, the statement says that the last (and first!) $2^\nn$ values in the row of~$2^\nn$ are $1, 2 \wdots 2^\nn$, which is indeed true.
\end{proof}

We now refine Lemma~\ref{L:Valuation} by proving that $\qq$ may appear in the column of~$\qq$ only at the positions previously described.

\begin{lemm}
\label{L:ValuationPlus}
Assume $1 \le \pp \le 2^\nn$ and $2^\dd \div \pp$ with $2^{\dd+1} \notdiv \pp$. Then $\pp \OP\nn \qq = \qq$ holds for $2^\nn - 2^\dd < \qq \le 2^\nn$ and fails for $1 \le \qq \le 2^\nn - 2^\dd$. In particular, for $\qq \le 2^{\nn-1}$, the equality $\pp \OP\nn \qq = \qq$ holds for $\pp = 2^\nn$ only.
\end{lemm}

\begin{proof}
First we note that the result is obvious for $\pp = 2^\nn$, and from now we restrict to the case $\pp < 2^\nn$, implying $\dd \le \nn-1$. We use induction on~$\nn$. The result is true in~$\AA_0$ and~$\AA_1$. Assume $\nn \ge 2$ and $1 \le \pp < 2^\nn$ with $2^\dd \div \pp$ and $2^{\dd+1} \notdiv \pp$. We inspect the $\pp$th row in the table of~$\AA_\nn$, that is, the values $\pp \OP\nn \qq$ for $1 \le \qq \le 2^\nn$. Let $\bar\pp = \MOD\pp{2^{\nn-1}}$ and $\bar\qq = \MOD\qq{2^{\nn-1}}$. We consider two cases.

Assume first $1 \le \bar\qq \le 2^{\nn-1} - 2^\dd$. This case  may arise only for~$\dd < \nn-1$.  Consequently, we have $2^\dd \div2^{\dd+1} \div 2^{\nn-1}$, so the conjunction of  $2^\dd \div \pp$ and $2^{\dd+1} \notdiv \pp$ implies $2^\dd \div \bar\pp$ and $2^{\dd+1} \notdiv \bar\pp$. Then the induction hypothesis  applies, thus  giving $\bar\pp \OP{\nn-1} \bar\qq \not= \bar\qq$, which implies $\pp \OP\nn \qq \not= \qq$ by Proposition~\ref{P:Proj}.

Assume now $2^{\nn-1} - 2^\dd < \bar\qq \le 2^{\nn-1}$, that is, either   \ITEM1 $2^{\nn-1} - 2^\dd < \qq \le 2^{\nn-1}$ or \ITEM2 $2^\nn - 2^\dd < \qq \le 2^\nn$. In case~\ITEM2, Lemma~\ref{L:Valuation} implies $\pp \OP\nn \qq = \qq$, as expected. Now, as $\pp$ is not~$2^\nn$, the period~$\per_\nn(\pp)$ is at most~$2^{\nn-1}$ and, therefore, in case~\ITEM1, we have $\pp \OP\nn \qq = \pp \OP\nn (\qq + 2^{\nn-1})$, whence $\pp \OP\nn \qq = \qq + 2^{\nn-1} \not= \qq$ by Lemma~\ref{L:Valuation} again. Thus $\pp \OP\nn \qq = \qq$ holds exactly in the expected cases.
\end{proof}

As we now control the occurrences of the value~$\qq$ in the column of~$\qq$ precisely, we can easily deduce that  the columns are pairwise distinct in a Laver table (Theorem~1.8 of~\cite{DraSem}).

\begin{lemm}
\label{L:Distinct}
For every~$\nn$, the families~$\Col_\nn(\qq)$ are pairwise distinct when $\qq$ ranges over $\{1 \wdots 2^\nn\}$. 
\end{lemm}

\begin{proof}
We shall use induction on~$\nn$. The result is true in~$\AA_0$ (vacuously) and~$\AA_1$. Assume $\nn \ge 2$. We shall analyze $\Col_\nn(\qq) \cap \{2^{\nn-1}+1 \wdots 2^\nn\}$, that is, inspect the large values in the column of~$\qq$. So, consider $1 \le \pp, \qq \le 2^\nn$. As usual, we put $\bar\pp = \MOD\pp{2^{\nn-1}}$ and $\bar\qq = \MOD\qq{2^{\nn-1}}$. By Proposition~\ref{P:Proj}, we have
$$\pp \ \OP\nn\  \qq = \bar\pp \ \op_{\nn-1}\  \bar\qq + \eps 2^{\nn-1}$$
with $\eps$ in~$\{0, 1\}$; moreover, we know that all values in the row of~$\pp$ are larger than~$\pp$ for $\pp \neq 2^\nn$, so $\eps$ must be~$1$ for $2^\nn > \pp > 2^{\nn-1}$. Together with
$$2^{\nn - 1} \ \OP\nn\  \qq = \bar\qq +  2^{\nn-1} = 2^{\nn - 1} \ \op_{\nn-1}\  \bar\qq +  2^{\nn-1},$$
this implies, with obvious notation,
\begin{equation}
\label{E:Distinct1}
\Col_\nn(\qq) \cap \{2^{\nn-1}+1 \wdots 2^\nn\} = \Col_{\nn-1}(\bar\qq) + 2^{\nn-1}.
\end{equation}
Now assume $1 \le \qq < \rr \le 2^\nn$. We write $\bar\rr = \MOD\rr{2^{\nn-1}}$, and consider two cases.

Assume first $\bar\qq \not= \bar\rr$. Then the induction hypothesis implies $\Col_{\nn-1}(\bar\qq) \not= \Col_{\nn-1}(\bar\rr)$, and, by \eqref{E:Distinct1}, we deduce $\Col_\nn(\qq) \cap \{2^{\nn-1}+1 \wdots 2^\nn\} \not= \Col_\nn(\rr) \cap \{2^{\nn-1}+1 \wdots 2^\nn\}$, whence $\Col_\nn(\qq) \not= \Col_\nn(\rr)$ \textit{a fortiori}.

Assume now $\bar\qq = \bar\rr$. Then we necessarily have $\rr = \qq + 2^{\nn-1}$, and \eqref{E:Distinct1} only says that the intersections of $\Col_\nn(\qq)$ and $\Col_\nn(\rr)$ with $\{2^{\nn-1}+1 \wdots 2^\nn\}$ coincide. Now look at the $\pp$th row with $\pp < 2^\nn$, (the final statement in) Lemma~\ref{L:ValuationPlus} tells us that we have $\pp \OP\nn \qq \not= \qq$, which also implies $\pp \OP\nn \rr \not= \qq$ since $\per_\nn(\pp)$ divides~$2^{\nn-1}$. In other words, the value~$\qq$ cannot appear in the columns of~$\qq$ and~$\rr$, except possibly in the row of~$2^\nn$. Now we have $2^\nn \OP\nn \qq = \qq$, and $2^\nn \OP\nn \rr = \rr \not= \qq$, so $\Col_\nn(\qq)$ and $\Col_\nn(\rr)$ as distinct (and more precisely,  using $\sqcup$ for disjoint union, we have $\Col_\nn(\qq) = \Col_\nn(\rr) \sqcup \{\qq\}$ since the periods of all rows except the last one divides~$\rr -\nobreak \qq$, which is  $2^{\nn-1}$).
\end{proof}

Note that, in the proof, we have precisely determined how certain columns differ---namely, for all~$\qq \le 2^{\nn-1}$, one has
\begin{equation}
\label{E:DistinctPrecise}
\Col_\nn(\qq) = \Col_\nn( \qq + 2^{\nn-1}) \sqcup \{\qq\}.
\end{equation}

The last preliminary result before proving Proposition~\ref{P:Poset} is a (nontrivial) connection between the operation~$\op$ and an associative operation. Well known to R.\,Laver, the result is implicit in~\cite{Lvd}, and explicit for instance in~\cite{DouJech} and~\cite{DraSem}---see also \cite[Proposition~XI.2.15]{Dgd}, where the result comes as an application of general facts about what is called LD-monoids. 

\begin{lemm}
\label{L:Monoid}
For every~$\nn$, there exists a unique binary operation $\comp$ on~$\{1 \wdots 2^\nn\}$ such that the law
\begin{equation}
\label{E:Monoid}
(\xx \comp \yy) \op \zz = \xx \op (\yy \op \zz),
\end{equation}
is obeyed, namely the operation~$\comp_\nn$ defined by
\begin{equation}
\label{E:DefComp}
\pp \comp_\nn \qq = \begin{cases}
\pp \OP\nn (\qq+1) -1 &\mbox{for $\qq < 2^n$},\\
\pp &\mbox{for $\qq = 2^\nn$}
\end{cases}
\end{equation}
The operation~$\comp_n$ is associative, and it admits~Ê$2^\nn$ as a neutral element.
\end{lemm}

\begin{proof}
Let us first observe that the two cases in~\eqref{E:DefComp} merge into the single formula
\begin{equation}
\label{E:DefCompBis}
\pp \comp_\nn \qq + 1 = \pp \OP\nn (\qq+1) \MOD{}{2^\nn},
\end{equation}
which is equivalent to 
\begin{equation}
\label{E:Monoid1}
(\pp \comp_\nn \qq) \OP\nn 1 = \pp \OP\nn (\qq \OP\nn 1).
\end{equation}
As~\eqref{E:Monoid1} is a particular case of~\eqref{E:Monoid}, we deduce that $\comp_\nn$ is the unique operation possibly obeying~\eqref{E:Monoid}.

Next, for $1 \le \pp \le 2^\nn$, let $\ad_\pp$ be the left-translation of~$\AA_\nn$ associated with~$\pp$, that is, the function $\yy \mapsto\nobreak \pp \OP\nn \yy$. The form of the left-selfdistributivity law implies that every map~$\ad_\pp$ is an endomorphism of~$\AA_\nn$ as we have
$$\ad_\pp (\xx \OP\nn \yy) = \pp \OP\nn (\xx \OP\nn \yy) = (\pp \OP\nn \xx) \OP\nn (\pp \OP\nn \yy) = \ad_\pp(\xx) \OP\nn \ad_\pp(\yy)$$
(a property that is not specific to Laver tables). 

Now, \eqref{E:Monoid1} rewrites into $\ad_{\pp \comp_\nn \qq}(1) = \ad_\pp(\ad_\qq (1))$. As each of $\ad_{\pp \comp_\nn \qq}$, $\ad_\pp$, and $\ad_\qq$ is an endomorphism of~$\AA_\nn$ and $1$ generates~$\AA_\nn$, the equality of~$\ad_{\pp \comp_\nn \qq}$ and $\ad_\pp \comp \ad_\qq$ on~$1$ implies their equality everywhere, that is, for all $\pp, \qq$, we have 
\begin{equation}
\label{E:Comp}
\ad_{\pp \comp_\nn \qq} = \ad_\pp \comp \ad_\qq.
\end{equation}
This means that $(\pp \comp_\nn \qq) \OP\nn \rr = \pp \OP\nn (\qq \OP\nn \rr)$ holds for all~$\pp, \qq, \rr$, that is, the operations~$\OP\nn$ and~$\comp_\nn$ obey the law~\eqref{E:Monoid}. Thus the existence part of the result is  also  established.

We turn to the associativity of~$\comp_\nn$, which follows from that of composition. Indeed, for all $\pp, \qq, \rr$, \eqref{E:Comp} gives
$$\ad_{\pp \comp_\nn (\qq \comp_\nn \rr)} = \ad_\pp \comp (\ad_\qq \comp \ad_\rr) = (\ad_\pp \comp \ad_\qq) \comp \ad_\rr = \ad_{(\pp \comp_\nn \qq) \comp_\nn \rr}.$$
We deduce in particular $\ad_{\pp \comp_\nn (\qq \comp_\nn \rr)}(1) = \ad_{(\pp \comp_\nn \qq) \comp_\nn \rr}(1)$, which implies $\pp \comp_\nn (\qq \comp_\nn \rr) = (\pp \comp_\nn \qq) \comp_\nn \rr$ since the function $\xx \mapsto \xx \OP\nn 1$ is injective on~$\AA_\nn$.

Finally,  Relations~\eqref{E:DefComp}  and~\eqref{E:LastRows} imply that $2^\nn$ is a neutral element in~$(\AA_\nn, \comp_\nn)$. 
\end{proof}

We refer to Table~\ref{T:Comp} for some examples of the associative operations~$\comp_\nn$. All periodicity phenomena involving the operation~$\OP\nn$ also appear in the table of~$\comp_\nn$. Observe that the monoid $(\AA_\nn, \comp_\nn)$ is not monogenerated for $\nn \ge 2$. It can be mentioned that two other laws connect the operations~$\OP\nn$ and~$\comp_\nn$, namely
$$\xx \comp_\nn \yy = (\xx \OP\nn \yy) \comp_\nn \xx \mbox{\quad and \quad} \xx \OP\nn (\yy \comp_\nn \zz) = (\xx \OP\nn \yy) \comp_\nn (\xx \OP\nn \zz),$$
but these relations will not be used here. Because the operation~$\comp_\nn$ is directly defined from the selfdistributive operation~$\OP\nn$, it adds nothing really new to the structure of Laver tables. 

\begin{table}
\small\begin{tabular}{c|c}
$\comp_0$&$1$\\
\hline
$1$&$1$
\end{tabular}
\quad 
\begin{tabular}{c|cc}
$\comp_1$&$1$&$2$\\
\hline
$1$&$1$&$1$\\
$2$&$1$&$2$\\
\end{tabular}
\quad 
\begin{tabular}{c|cccc}
$\comp_2$&$1$&$2$&$3$&$4$\\
\hline
$1$&$3$&$1$&$3$&$1$\\
$2$&$3$&$2$&$3$&$2$\\
$3$&$3$&$3$&$3$&$3$\\
$4$&$1$&$2$&$3$&$4$\\
\end{tabular}
\quad 
\begin{tabular}{c|cccccccc}
$\comp_3$&$1$&$2$&$3$&$4$&$5$&$6$&$7$&$8$\\
\hline
$1$&$3$&$5$&$7$&$1$&$3$&$5$&$7$&$1$\\
$2$&$3$&$6$&$7$&$2$&$3$&$6$&$7$&$2$\\
$3$&$7$&$3$&$7$&$3$&$7$&$3$&$7$&$3$\\
$4$&$5$&$6$&$7$&$4$&$5$&$6$&$7$&$4$\\
$5$&$7$&$5$&$7$&$5$&$7$&$5$&$7$&$5$\\
$6$&$7$&$6$&$7$&$6$&$7$&$6$&$7$&$6$\\
$7$&$7$&$7$&$7$&$7$&$7$&$7$&$7$&$7$\\
$8$&$1$&$2$&$3$&$4$&$5$&$6$&$7$&$8$\\
\end{tabular}
\vspace{3mm}
\caption{\sf\small The associative operation~$\comp_\nn$ on~$\AA_\nn$. The columns in the table of~$\comp_\nn$ are obtained by shifting the columns of~$\OP\nn$ by one unit to the left, and substracting~$1$ mod~$2^\nn$. In particular, the bottom row and the last column are identities, while the next to last row and column are constant with value~$2^\nn {-} 1$.}
\label{T:Comp} 
\end{table}

We are now ready to establish the existence of the expected ordering.

\begin{proof}[Proof of Proposition~\ref{P:Poset}]
We first show that the relation~$\Before\nn$ is transitive. So assume $\pp\Before\nn \qq \Before\nn \rr$. By definition, there exist~$\qq', \rr'$ satisfying $\qq' \OP\nn \pp = \qq$ and $\rr' \OP\nn \qq = \rr$. Now, applying Lemma~\ref{L:Monoid}, we obtain
$$(\rr' \comp_\nn \qq') \OP\nn \pp =  \rr' \OP\nn (\qq' \OP\nn \pp) = \rr,$$
whence $\pp\Before\nn \rr$.

Once we know that $\Before\nn$ is transitive, we deduce the equivalence
\begin{equation}
\label{E:Equiv}
\pp \Before\nn \qq \quad \Leftrightarrow \quad \Col_\nn(\pp)  \supseteq  \Col_\nn(\qq).
\end{equation}
Indeed, assume $\pp \Before\nn \qq$, and $\rr \in \Col_\nn(\qq)$: by definition, we have $\qq \Before\nn \rr$, whence $\pp\Before\nn\nobreak \rr$ by transitivity, and therefore  $\Col_\nn(\pp)$ contains~$\rr$.  So the relation $\pp \Before\nn \qq$ implies $\Col_\nn(\pp) \supseteq\nobreak \Col_\nn(\qq)$. Conversely, we know that $\qq$ always belongs to~$\Col_\nn(\qq)$, so $\Col_\nn(\pp) \supseteq \Col_\nn(\qq)$ implies that $\Col_\nn(\pp)$ contains~$\qq$, which is $\pp\Before\nn \qq$.

We now deduce that the relation~$\Before\nn$ is antisymmetric. Indeed, assume $\pp \Before\nn \qq$ and $\qq \Before\nn \pp$. Applying~\eqref{E:Equiv}, we deduce $\Col_\nn(\pp) = \Col_\nn(\qq)$, whence $\pp = \qq$ by Lemma~\ref{L:Distinct}. So $\Before\nn$ is a partial ordering on~$\{1 \wdots \nn\}$.

Next, we have $\Col_\nn(1) = \{1 \wdots 2^\nn\}$, hence $1 \Before\nn \pp$ is always satisfied.  On the other hand, for every~$\pp$, we have $(2^\nn-1)\OP\nn \pp = 2^\nn$ due to~\eqref{E:LastRows}, and $\Col_\nn(2^\nn) = \{2^\nn\}$ due to~\eqref{E:LastColumn}, hence $\Col_\nn(\pp) \supseteq \Col_\nn(2^\nn)$ for every~$\pp$, meaning that $\pp \Before\nn 2^\nn$ is always satisfied.
\end{proof}

For every~$\nn \ge 1$, the partial order~$\Before{\nn-1}$ can be obtained from~$\Before\nn$ easily: indeed Relation~\eqref{E:Distinct1} implies that, for $\qq, \rr \le 2^{\nn-1}$, we have
\begin{equation}
\label{E:ProjOrder}
\rr \Before{\nn-1} \qq \quad\Leftrightarrow\quad (\rr + 2^{\nn-1}) \Before\nn (\qq + 2^{\nn-1}).
\end{equation}
In the other direction, recovering~$\Before\nn$ from~$\before_{\nn-1}$ is more problematic: \eqref{E:ProjOrder} determines~$\Before\nn$ on $\{2^{\nn-1}+1 \wdots 2^\nn\}$, and Relation~\eqref{E:DistinctPrecise} implies that,  for $\qq \le 2^{\nn-1}$,
\begin{equation}
\label{E:ProjOrder2}
\qq \Before\nn \qq + 2^{\nn-1}
\end{equation}
holds and, more precisely, that $\qq$ is an immediate predecessor of~$\qq +\nobreak 2^{\nn-1}$, but this says nothing about the way the predecessors of~$\qq + 2^{\nn-1}$ lying in $\{2^{\nn-1}+\nobreak 1 \wdots 2^\nn\}$ connect with~$\qq$. Figure~\ref{F:Poset} shows that these connections can take very different forms. However, the analysis can be completed in the particular case of the final elements of~$(\AA_\nn, \Before\nn)$:

\begin{prop}
\label{P:MinusOne}
For $\nn \ge 2$ and $1 \le \pp < 2^\nn$ with $\pp \not= 2^{\nn-1}$, we have 
\begin{equation}
\label{E:MinusOne}
\pp \Before\nn 2^\nn - 2^{\nn-2} \Before\nn 2^{\nn-1} \Before\nn 2^\nn.
\end{equation}
\end{prop}

\begin{proof}
We use induction on~$\nn \ge 2$. For $\nn = 2$, we have $1 \Before2 3 \Before2 2 \Before2 4$, and the result is true. Assume $\nn \ge 3$.  First, we know already that $2^{\nn-1} \Before\nn 2^\nn$ holds. Next, consider the row of~$2^{\nn-2}$ in~$\AA_\nn$. By Lemma~\ref{L:ValuationPlus}, the last values are $2^{\nn-2}$ consecutive numbers increasing from $2^\nn - 2^{\nn-2} + 1$ to~$2^\nn$, whereas the first value is~$2^{\nn-2}+1$: hence, at least $2^{\nn-2}+1$ different values appear on this row, and we deduce $\per_\nn(2^{\nn-2}) \ge 2^{\nn-1}$, whence  $\per_\nn(2^{\nn-2})  = 2^{\nn-1}$ since $\per_\nn(2^{\nn-2}) = 2^\nn$ would require $2^{\nn-2} \OP\nn 1 = 1$. We deduce $2^{\nn-2} \OP\nn (2^\nn - 2^{\nn-2}) < 2^\nn$, which, owing to the obvious equality $2^{\nn-2} \OP{\nn-1} 2^{\nn-2} = 2^{\nn-1}$, implies
\begin{equation}
\label{E:MinusTwo}
2^{\nn-2} \OP\nn (2^\nn - 2^{\nn-2}) = 2^{\nn-1}.
\end{equation}
So $2^{\nn-1}$ appears in the column of $2^\nn - 2^{\nn-2}$, and the relation $2^\nn - 2^{\nn-2} \Before\nn 2^{\nn-1}$ is satisfied.

Now consider $\pp < 2^\nn$, $\pp \not= 2^{\nn-1}$. Assume first $\pp > 2^{\nn-1}$, and let $\bar\pp = \MOD\pp{2^{\nn-1}}$. Then $\bar\pp$ is not~$2^{\nn-1}$ so, by induction hypothesis, we have $\bar\pp \Before{\nn-1} 2^{\nn-2}$. By~\eqref{E:ProjOrder}, we deduce $\bar\pp + 2^{\nn-1} \Before\nn 2^{\nn-2} + 2^{\nn-1}$, which is also $\pp \Before\nn 2^\nn - 2^{\nn-2}$, as expected. Assume now $\pp < 2^{\nn-1}$. Then \eqref{E:ProjOrder2} implies $\pp \Before\nn \pp + 2^{\nn-1}$. We proved above $\pp + 2^{\nn-1} \Before\nn 2^{\nn-2} + 2^{\nn-1}$. Using the transitivity of~$\Before\nn$, we deduce $\pp \Before\nn 2^\nn - 2^{\nn-2}$ again.
\end{proof}

The previous result is, in a sense, optimal. Indeed, Proposition~\ref{P:MinusOne} says that the relation~$\Before\nn$ always admits a length~$3$ tail made of $2^\nn - 2^{\nn-2}$, $2^{\nn-1}$, and~$2^\nn$, but we cannot expect more:  with respect to~$\before_3$,  the last three elements are $6$, $4$, and~$8$, but $6$ has two incomparable immediate predecessors, namely~$2$ and~$7$.

Finally, at the other end, the fact that every period but the last one in~$\AA_\nn$ divides~$2^{\nn-1}$ implies that $\Col_\nn(2^{\nn-1} {+} 1)$ is $\{2 \wdots 2^\nn\}$, hence it contains every element but~$1$. Thus $2^{\nn-1}{+}1 \Before\nn \pp$ holds for $2 \le \pp \le 2^\nn$, again an optimal result as, in~$\AA_3$, the element~$5$ admits two incomparable immediate successors, namely~$2$ and~$3$ (as can be seen in Figure~\ref{F:Poset}).

 To conclude the section, we point out another application of Relations~\eqref{E:Distinct1} and~\eqref{E:ProjOrder}.

\begin{coro}
\label{C:Occurrences}
Assume $\mm \le \nn$, $1 \le \qq \le 2^\nn$, and $1 \le \rr < 2^\mm$. Then $2^\nn - \rr$ appears in the $\qq$th column of~$\AA_\nn$ if and only if $2^\mm - \rr$ appears in the $\bar\qq$th column of~$\AA_\mm$, with $\bar\qq = \MOD\qq{2^\mm}$.
\end{coro}

\begin{proof}
 Owing to~\eqref{E:ProjOrder}, the result follows from a straightforward induction on~$\nn -\nobreak \mm$. 
\end{proof}

\begin{exam}
As $1$ occurs in the first column of~$\AA_1$ and not in the second one, we deduce that 
\begin{equation}
\label{E:Constraint}
\mbox{$2^\nn - 1$ occurs in the $\qq$th column of~$\AA_\nn$ exactly for $\qq = 1 \MOD{}2$.}
\end{equation}
Similarly looking at the occurrences of $4-\rr$ in~$\AA_2$, at those of $8 - \rr$ in~$\AA_3$, \textit{etc}. leads to the constraints indicated in the following array, where we use $\MODD\qq{\mm}$ for~$\MOD\qq{\mm}$, the meaning being that $2^\nn - \rr$ occurs in the $\qq$th column of~$\AA_\nn$ if and only if the corresponding constraint is satisfied by~$\qq$:
$$\begin{tabular}{c|c|c|c|c|c|c|c|c}
$\rr$&$1$&$2$&$3$&$4$&$5$&$6$&$7$&$8$\\
\hline
\VR(4,0)$\qq$&$= \MODD12$&$\not= \MODD44$&$= \MODD14$&$\not= \MODD88$&$= \MODD{1,3,5}8$&$= \MODD{1,2,5}8$&$= \MODD18$&$\not= \MODD{16}{16}$
\end{tabular}$$

Another way of stating~\eqref{E:Constraint} is the equivalence
\begin{equation}
\label{E:Constraint1}
\qq \Before\nn 2^\nn - 1 \mbox{\quad $\Longleftrightarrow$ \quad} \qq = 1 \MOD{}2
\end{equation}
(and similarly for every value of~$\rr$); in this equivalence, the left-to-right implication follows from Corollary~\ref{C:Odd},  but the converse implication is less trivial. In particular, we note that the map $\pp \mapsto 2\pp - 1$ need not be increasing from $(\AA_{\nn-1}, \Before{\nn-1})$ to~$(\AA_\nn, \Before\nn)$: for instance, $3 \before_3 2$ fails (as does $2 \before_3 3$) but $5 \before_4 3$ holds. This illustrates the fact that the restriction of the partial ordering~$\Before\nn$ to~$\{1 \wdots 2^{\nn-1}\}$ does not follow from the partial ordering~$\Before{\nn-1}$ simply.
\end{exam}

\section{An enhanced basis for $2$-cocycles}
\label{S:TwoPsi}

We now return to the investigation of $2$-cocycles for Laver tables, and use the partial ordering of Section~\ref{S:Poset} to exhibit an alternative basis for the abelian group~$\ZZ^2(\AA_\nn)$. From a combinatorial point of view, this new basis seems more interesting in that it consists of cocycles that only take values~$0$ and~$1$ on~$\AA_\nn^2$, contrary to the basis of Proposition~\ref{P:Basis1} in which negative values appear. 

\begin{prop}
\label{P:Basis2}
For $1 \le \qq \le 2^\nn$ define $\psi_{\qq, \nn} = \sum_{\rr \Before\nn \qq} \phi_{\rr, \nn}$. Then the family $\{\psi_{1, \nn} \wdots \psi_{2^\nn-1, \nn}, \const_\nn\}$ is a basis of~$\ZZ^2(\AA_\nn)$,  and we have 
\begin{equation}
\label{E:Basis2}
\psi_{\qq, \nn}(\xx, \yy) = \begin{cases}
\ 1&\mbox{if $\qq$ belongs to~$\Col_\nn(\yy)$ but not to~$\Col_\nn(\xx \OP\nn \yy)$},\\
\ 0&\mbox{otherwise}.\end{cases}
\end{equation}
\end{prop}

\begin{proof}
By Proposition~\ref{P:Basis1}, $\{\phi_{1, \nn} \wdots \phi_{2^\nn-1, \nn}, \const_\nn\}$ is a basis of~$\ZZ^2(\AA_\nn)$. Now, the fact that, according to Proposition~\ref{P:Poset}, the relation~$\Before\nn$ is a partial order on~$\{1 \wdots 2^\nn\}$ implies that the families $\{\phi_{1, \nn} \wdots \phi_{2^\nn-1, \nn}, \const_\nn\}$ and $\{\psi_{1, \nn} \wdots \psi_{2^\nn-1, \nn}, \const_\nn\}$ are connected by a triangular matrix with diagonal entries equal to~$1$ (and all entries equal to~$0$ or~$1$): indeed, such a matrix arises whenever $\{1 \wdots 2^\nn\}$ is re-ordered according to any linear ordering that extends~$\Before\nn$. Hence $\{\psi_{1, \nn} \wdots \psi_{2^\nn-1, \nn}, \const_\nn\}$ is a basis of~$\ZZ^2(\AA_\nn)$---and $\{\psi_{1, \nn} \wdots \psi_{2^\nn-1, \nn}\}$  is a basis of~$\BB^2(\AA_\nn)$.    

Let us turn to the values of the cocycles~$\psi_{\qq, \nn}$. For  every $\qq$, we have by definition
\begin{equation}
\label{E:Basis2'}
\psi_{\qq, \nn}(\xx, \yy) = \sum_{\rr \Before\nn \qq} \delta_{\rr, \yy} - \sum_{\rr \Before\nn \qq} \delta_{\rr, \xx \op \yy}. 
\end{equation}
Clearly, at most one value of~$\rr$ may contribute to each sum in~\eqref{E:Basis2'}, so the value can only be $-1$, $0$, or~$1$. Now, assume that the contribution of the second sum is~$-1$. This means that we have $\xx \OP\nn \yy = \rr$ for some~$\rr$ satisfying $\rr \Before\nn \qq$, hence that $\xx \OP\nn \yy \Before\nn \qq$ holds. By definition, we have $\yy \Before\nn \xx \OP\nn \yy$. Since $\Before\nn$ is transitive, we deduce $\yy \Before\nn \qq$,  implying  that the contribution of the first sum is~$+1$.  So,  the global value of $\psi_{\qq, \nn}(\xx, \yy)$ cannot be negative.

Now, Relation~\eqref{E:Basis2} is clear since $\psi_{\qq, \nn}(\xx, \yy)$ is $1$ exactly when the contribution of the first sum in~\eqref{E:Basis2'} is~$1$, that is, when $\yy \Before\nn \qq$ is true, and that of the second sum is~$0$, that is, when $\xx \OP\nn \yy \Before\nn \qq$ is false.
\end{proof}

Note that, because $\rr \Before\nn 2^\nn$ is true for every~$\rr$, \eqref{E:Basis2'} implies that $\psi_{2^\nn, \nn}$ is the zero cocycle. We refer to Table~\ref{T:Psi} for a list of the seven nontrivial  cocycles $\psi_{\qq, 3}$ corresponding to~$\AA_3$.

\begin{table}[htb]
\def\arrayitem{2.4mm}
\small\small
\begin{tabular}{c|c}
\VR(0,1.5)\hidewidth{$\psi_{1,3}$}&\1\2\3\4\5\6\7\8\\
\hline
\VR(3,0)$1$&\1\0\0\0\0\0\0\0\\
$2$&\1\0\0\0\0\0\0\0\\
$3$&\1\0\0\0\0\0\0\0\\
$4$&\1\0\0\0\0\0\0\0\\
$5$&\1\0\0\0\0\0\0\0\\
$6$&\1\0\0\0\0\0\0\0\\
$7$&\1\0\0\0\0\0\0\0\\
$8$&\0\0\0\0\0\0\0\0
\end{tabular}
\quad
\begin{tabular}{c|c}
\VR(0,1.5)\hidewidth{$\psi_{2,3}$}&\1\2\3\4\5\6\7\8\\
\hline
\VR(3,0)$1$&\0\1\0\0\0\0\0\0\\
$2$&\1\1\0\0\1\0\0\0\\
$3$&\1\1\0\0\1\0\0\0\\
$4$&\0\1\0\0\0\0\0\0\\
$5$&\1\1\0\0\1\0\0\0\\
$6$&\1\1\0\0\1\0\0\0\\
$7$&\1\1\0\0\1\0\0\0\\
$8$&\0\0\0\0\0\0\0\0
\end{tabular}
\quad
\begin{tabular}{c|c}
\VR(0,1.5)\hidewidth{$\psi_{3,3}$}&\1\2\3\4\5\6\7\8\\
\hline
\VR(3,0)$1$&\1\0\1\0\1\0\0\0\\
$2$&\0\0\1\0\0\0\0\0\\
$3$&\1\0\1\0\1\0\0\0\\
$4$&\0\0\1\0\0\0\0\0\\
$5$&\1\0\1\0\1\0\0\0\\
$6$&\1\0\1\0\1\0\0\0\\
$7$&\1\0\1\0\1\0\0\0\\
$8$&\0\0\0\0\0\0\0\0
\end{tabular}
\quad
\begin{tabular}{c|c}
\VR(0,1.5)\hidewidth{$\psi_{4,3}$}&\1\2\3\4\5\6\7\8\\
\hline
\VR(3,0)$1$&\0\0\0\1\0\0\0\0\\
$2$&\0\0\0\1\0\0\0\0\\
$3$&\0\1\0\1\0\1\0\0\\
$4$&\0\0\0\1\0\0\0\0\\
$5$&\0\1\0\1\0\1\0\0\\
$6$&\0\1\0\1\0\1\0\0\\
$7$&\1\1\1\1\1\1\1\0\\
$8$&\0\0\0\0\0\0\0\0
\end{tabular}

\VR(1,0)

\begin{tabular}{c|c}
\VR(0,1.5)\hidewidth{$\psi_{5,3}$}&\1\2\3\4\5\6\7\8\\
\hline
\VR(3,0)$1$&\1\0\0\0\1\0\0\0\\
$2$&\1\0\0\0\1\0\0\0\\
$3$&\1\0\0\0\1\0\0\0\\
$4$&\0\0\0\0\0\0\0\0\\
$5$&\1\0\0\0\1\0\0\0\\
$6$&\1\0\0\0\1\0\0\0\\
$7$&\1\0\0\0\1\0\0\0\\
$8$&\0\0\0\0\0\0\0\0
\end{tabular}
\quad
\begin{tabular}{c|c}
\VR(0,1.5)\hidewidth{$\psi_{6,3}$}&\1\2\3\4\5\6\7\8\\
\hline
\VR(3,0)$1$&\0\1\0\0\0\1\0\0\\
$2$&\0\1\0\0\0\1\0\0\\
$3$&\1\1\1\0\1\1\1\0\\
$4$&\0\0\0\0\0\0\0\0\\
$5$&\0\1\0\0\0\1\0\0\\
$6$&\0\1\0\0\0\1\0\0\\
$7$&\1\1\1\0\1\1\1\0\\
$8$&\0\0\0\0\0\0\0\0
\end{tabular}
\quad
\begin{tabular}{c|c}
\VR(0,1.5)\hidewidth{$\psi_{7,3}$}&\1\2\3\4\5\6\7\8\\
\hline
\VR(3,0)$1$&\1\0\1\0\1\0\1\0\\
$2$&\0\0\0\0\0\0\0\0\\
$3$&\1\0\1\0\1\0\1\0\\
$4$&\0\0\0\0\0\0\0\0\\
$5$&\1\0\1\0\1\0\1\0\\
$6$&\0\0\0\0\0\0\0\0\\
$7$&\1\0\1\0\1\0\1\0\\
$8$&\0\0\0\0\0\0\0\0
\end{tabular}\\
\VR(1,0)
\caption{\sf\small A basis of $\BB^2(\AA_3)$ consisting of seven $\{0, 1\}$-valued $2$-cocycles; completing with the constant cocycle~$\const_3$, we obtain a basis of~$\ZZ^2(\AA_3)$. Note that, according to Lemma~\ref{L:PenultRow2}, the $7$th rows of the above cocycles must be pairwise distinct.}
\label{T:Psi}
\end{table}

 By construction, the cocycles~$\psi_{\qq, \nn}$ are coboundaries. Relation~\eqref{E:Basis2'} makes it straighforward that we have $\psi_{\qq, \nn} = - \der^2\gamma_{\qq, \nn}$, where $\gamma_{\qq, \nn}$ is the $1$-cochain defined on~$\AA_\nn$ by $\gamma_{\qq, \nn}(\xx) = 1$ if $\xx \Before\nn \qq$ is true, and $\gamma_{\qq, \nn}(\xx) = 0$ otherwise.

It turns out that $2$-cocycles contain a lot of information about the structure of Laver tables---certainly a promising point in view of potential applications. 

\begin{prop}
\label{P:PeriodCocycle}
For every~$\nn$, the $2$-cocycle~$\psi_{2^{\nn-1}, \nn}$, which is also $- \phi_{2^\nn, \nn}$, encodes periods in~$\AA_\nn$ in the sense that, for every~$\pp < 2^\nn$, the value of~$\per_\nn(\pp)$ is the smallest~$\yy$ satisfying $\psi_{2^{\nn-1}, \nn}(\pp, \yy) = 1$. 
\end{prop}

\begin{proof}
By definition, we have $\psi_{2^{\nn-1}, \nn} = \sum_{\rr \Before\nn 2^{\nn-1}} \phi_{\rr, \nn}$, whence, by Proposition~\ref{P:MinusOne},
\begin{equation}
\label{E:PeriodCocycle}
\psi_{2^{\nn-1}, \nn} = \sum_{1 \le \rr < 2^\nn} \phi_{\rr, \nn}.
\end{equation}
On the other hand, we noted just after the proof of Proposition~\ref{P:Basis2} that $\psi_{2^\nn, \nn}$, which is the sum of all cocycles $\phi_{\pp, \nn}$, is the zero cocycle. Comparing with~\eqref{E:PeriodCocycle}, we deduce $\psi_{2^{\nn-1}, \nn} = - \phi_{2^\nn, \nn}$. Now, by definition of~$\phi_{\pp, \nn}$, we see that the equality $\phi_{2^\nn, \nn}(\pp, \yy) = -1$, that is, $\psi_{2^{\nn-1}, \nn}(\pp, \yy) = 1$, is equivalent, for $\yy < 2^\nn$, to $\pp \OP\nn \yy = 2^\nn$, hence to $\per_\nn(\pp) \div \yy$. 
\end{proof}

\begin{prop}
\label{P:ThresholdCocycle}
For every~$\nn$, let $\theta_\nn = \sum_{\rr = 1}^{\rr = 2^{\nn-1}} \phi_{\rr, \nn}$. Then the $2$-cocycle~$\theta_\nn$ encodes theresholds in~$\AA_\nn$ in the sense that, for every~$\pp <  2^{\nn-1}$,  the value of~$\thres_\nn(\pp) +\nobreak 1$ is the smallest integer~$\yy$ satisfying $\theta_{\nn}(\pp, \yy) = 1$. 
\end{prop}

\begin{proof}
 Assume $\pp < 2^{\nn-1}$.  By definition, the threshold $\thres_\nn(\pp)$ is the largest integer~$\qq$ such that, for all~$\qq' \le \qq$, one has $\pp \OP\nn \qq' = \pp \OP{\nn-1} \qq'$, or, equivalently, $\pp \OP\nn \qq' \le 2^{\nn-1}$, if such a~$\qq$ exists, and $0$ otherwise.  In other words, $\thres_\nn(\pp) + 1$ is the smallest integer~$\yy$ satisfying $\pp \OP\nn \yy > 2^{\nn-1}$. 

On the other hand, for $\yy \le 2^{\nn-1}$, we find
$$\theta_\nn(\pp, \yy) = \sum_{\rr = 1}^{2^{\nn-1}} \delta_{\yy, \rr} - \sum_{\rr = 1}^{2^{\nn-1}} \delta_{\pp \OP\nn \yy, \rr} = \begin{cases} \ 0 &\mbox{for $\pp \OP\nn \yy \le 2^{\nn-1}$},\\ \ 1 &\mbox{for $\pp \OP\nn \yy > 2^{\nn-1}$},\end{cases}$$
whence the result by merging the values and noting that $\thres_\nn(\pp) + 1 \le 2^{\nn-1}$  holds for every $\pp <  2^{\nn-1}$,  since one has $\pp  \OP\nn 2^{\nn-1} = 2^{\nn}  > 2^{\nn-1}$. 
\end{proof}

 The next easy observation is that the existence of the canonical projection~$\proj_\nn$ from~$\AA_\nn$ to~$\AA_{\nn-1}$ enables one to lift every $2$-cocycle on~$\AA_{\nn-1}$ into a $2$-cocycle on~$\AA_\nn$:

\begin{lemm}
\label{L:Cocycle2}
For every $2$-cocycle~$\phi$ on~$\AA_{\nn-1}$, define
\begin{equation}
\label{E:Cocycle2}
\proj_\nn^*(\phi)(\xx, \yy) = \phi(\xx \MOD{}{2^{\nn-1}}, \yy\MOD{}{2^{\nn-1}}). 
\end{equation}
Then $\proj_\nn^*(\phi)$ is a $2$-cocycle on~$\AA_\nn$.
\end{lemm}

The result is clear, as the projection~$\proj_\nn$ is a homomophism with respect to the operations~$\OP\nn$ and~$\OP{\nn-1}$. Note that the table of $\proj_\nn^*(\phi)$ is a $2^\nn \times 2^\nn$ square paved by four copies of the table of~$\phi$.

We should therefore be able to express in our distinguished bases of~$\ZZ^2(\AA_\nn)$ the lifted images of the cocycles of the corresponding distinguished bases of~$\ZZ^2(\AA_{\nn-1})$. This  is indeed  easy in the case of  the two  families considered so far. 

\begin{prop}
\label{P:Lifting}
For $\nn \ge 1$ and $1 \le \pp \le 2^{\nn-1}$, we have 
\begin{gather}
\label{E:Lifting1}
\proj_\nn^*(\phi_{\pp, \nn-1}) = \phi_{\pp, \nn} + \phi_{\pp + 2^{\nn-1}, \nn},\\
\label{E:Lifting2}
\proj_\nn^*(\psi_{\pp, \nn-1}) = \psi_{\pp + 2^{\nn-1}, \nn}.
\end{gather}
\end{prop}

\begin{proof}
Since the cocycles $\phi_{\pp, \nn}$ are defined in terms of the maps $\xx \mapsto  \delta_{\xx, \pp}$, let us describe the behaviour of the latter with respect to lifting. For $1 \le \pp \le 2^{\nn-1}$ and $1 \le \xx,\yy \le 2^{\nn}$, we have (putting as usual $\bar\xx = \MOD\xx{2^{\nn-1}}$)
$ \delta_{\bar\xx, \pp} = \delta_{\xx, \pp} + \delta_{\xx, \pp+ 2^{\nn-1}},$ therefore,
\begin{align*}
 \proj_\nn^*(\phi_{\pp, \nn-1})(\xx, \yy) &=  \delta_{\bar\yy, \pp} -  \delta_{\bar\xx \OP{\nn-1} \bar\yy,\pp} \\
 &= \delta_{\yy, \pp} + \delta_{\yy, \pp+ 2^{\nn-1}} - \delta_{\bar\xx \OP{\nn-1} \bar\yy, \pp} - \delta_{\bar\xx \OP{\nn-1} \bar\yy, \pp+ 2^{\nn-1}}\\
 &= \delta_{\yy, \pp} + \delta_{\yy, \pp+ 2^{\nn-1}} - \delta_{\xx \OP\nn \yy, \pp} - \delta_{\xx \OP\nn \yy, \pp+ 2^{\nn-1}} \\
 &=\phi_{\pp, \nn}(\xx, \yy) + \phi_{\pp + 2^{\nn-1}, \nn}(\xx, \yy),
 \end{align*}
 since $\xx \OP\nn \yy$  is either~$\bar\xx \OP{\nn-1} \bar\yy$ or~$\bar\xx \OP{\nn-1} \bar\yy+ 2^{\nn-1}$.

The liftings of  the cocycles~$\psi_{\qq, \nn}$  are now calculated using those of the  the cocycles~$\phi_{\pp, \nn}$. First, we find 
\begin{align*}
\proj_\nn^*(\psi_{\pp, \nn-1}) &= \sum_{\rr \le 2^{\nn-1},\rr \Before{\nn-1} \pp} \proj_\nn^*(\phi_{\rr, {\nn-1}}) = \sum_{\rr \le 2^{\nn-1},\rr \Before{\nn-1} \pp} (\phi_{\rr, \nn} + \phi_{\rr + 2^{\nn-1}, \nn})
\end{align*}
On the other hand,  we have 
\begin{align*}
\psi_{\pp + 2^{\nn-1}, \nn} &= \sum_{\rr \le 2^{\nn},\rr \Before{\nn} \pp+2^{\nn-1}} \phi_{\rr, {\nn}}.
\end{align*}
To conclude, it is sufficient to show that, for  all  $\rr \le 2^{\nn}$ and $\pp \le 2^{\nn-1}$,  the  condition $\rr \Before{\nn} \pp+2^{\nn-1}$ is equivalent to $\bar\rr \Before{\nn-1} \pp$.  Now,  due to~\eqref{E:ProjOrder}, the latter is the same as $\bar\rr+2^{\nn-1} \Before{\nn} \pp+2^{\nn-1}$. On the other hand, due to~\eqref{E:DistinctPrecise}, the element $\pp+2^{\nn-1}$,  which is larger than~$2^{\nn-1}$,  belongs either to the columns of both~$\bar\rr$ and~$\bar\rr+2^{\nn-1}$ in~$\AA_\nn$, or to neither of them. Thus $\bar\rr+2^{\nn-1} \Before{\nn} \pp+2^{\nn-1}$ is equivalent to $\rr \Before{\nn} \pp+2^{\nn-1}$, as desired. 
\end{proof}

So we see in particular that the family $\{\psi_{1, \nn} \wdots \psi_{2^\nn-1, \nn}\}$ consists of $2^{\nn-1}$ ``really new'' cocycles, plus $2^{\nn-1}-1$ cocycles that are liftings of cocycles on~$\AA_{\nn-1}$.

Of course, we can compose projections. For all $\mm \le \nn$, writing $\proj_{\nn, \mm}$ for the projection $\xx \mapsto \MOD\xx{2^\mm}$ from~$\AA_\nn$ to~$\AA_\mm$, we obtain a lifting $\proj_{\mm, \nn}^*$ from~$\ZZ^2(\AA_\mm)$ to~$\ZZ^2(\AA_\nn)$. Thus, for instance, among the $2^\nn{-}1$ cocycles~$\psi_{\pp, \nn}$, the last $2^{\nn-1}{-}1$ ones come from~$\ZZ^2(\AA_{\nn-1})$, among which the last $2^{\nn-2}{-}1$ ones come from~$\ZZ^2(\AA_{\nn-2})$, \textit{etc}. Also note that, trivially, the constant cocycles~$\const_\nn$ are liftings of one another: for all $\mm \le \nn$, we have $\const_\nn = \proj_{\nn, \mm}^*(\const_\mm)$.

\begin{exam}
Up to a multiplicative constant, there exists only one non-constant $2$-cocycle on~$\AA_1$, namely the cocycle~$\psi_{1, 1}$. It takes the value~$1$ for $\xx = \yy = 1$ only. Using Proposition~\ref{P:Lifting} repeatedly, we deduce for every~$\nn \ge 1$
\begin{equation}
\label{E:Parity}
\proj_{\nn,1}^*(\psi_{1, 1})  = \psi_{2^\nn-1, \nn}.
\end{equation}
By construction, this cocycle detects parity, in the sense that $\phi(\xx, \yy) = 1$ holds if and only if both~$\xx$ and $\yy$~are odd.
\end{exam}

 We showed above in Lemma~\ref{L:PenultRow2} that a $2$-cocycle for~$\AA_\nn$ is determined by its penultimate row. As a last observation, let us mention that it is also determined by its first column. Indeed, it is not hard to check that a $2$-cocycle whose first column is trivial must be the zero cocycle. For instance, the reader can check in Table~\ref{T:Psi} that the first columns of the cocycles~$\psi_{\qq, 3}$ with $1 \le \qq \le 7$ are linearly independent over~$\ZZZZ$.

\section{Three-cocycles for Laver tables}
\label{S:Three}

We conclude the paper with a study of $3$-cocycles and degree~$3$ rack cohomology for Laver tables.  According to~\eqref{E:Der4}, describing the $\ZZZZ$-valued $3$-cocycles for~$\AA_\nn$ amounts to  searching for maps $\phi: \{1 \wdots 2^\nn\}^3 \to \ZZZZ$ that satisfy 
\begin{align}
\label{E:Cocycle4}
\phi(\xx \op \yy, \xx \op \zz, \xx \op \tt) + &\phi(\xx, \yy, \zz \op \tt) + \phi(\xx, \zz, \tt)\\ 
\notag
&= \phi(\xx, \yy \op \zz, \yy \op \tt) + \phi(\yy, \zz, \tt) + \phi(\xx, \yy, \tt).
\end{align}
We keep omitting subscripts~$\Rack$ and coefficients~$\ZZZZ$ for brevity. 
 
We follow the same  scheme as for $2$-cocycles, omitting the details of the proofs when they are analogous to the $2$-cocycle case. First, let us describe basic $3$-coboundaries.
 
\begin{lemm}
\label{L:Delta3}
For $1 \le \pp, \qq \le 2^\nn$, define $\phi_{\pp,\qq, \nn}$ by 
\begin{equation}
\label{E:phi3}
\phi_{\pp,\qq, \nn}(\xx, \yy, \zz) = \delta_{\pp, \yy}\delta_{\qq, \zz} - \delta_{\pp,\xx \OP\nn \yy}\delta_{\qq, \xx \OP\nn \zz} - \delta_{\pp, \xx}\delta_{\qq, \zz} + \delta_{\pp,\xx}\delta_{\qq, \yy \OP\nn \zz}.
\end{equation}
Then, for all~$\pp,\qq$, the map~$\phi_{\pp,\qq, \nn}$ defines a $3$-coboundary on~$\AA_\nn$.
\end{lemm}

\begin{proof}
Let $\theta$ be the $2$-cochain on~$\AA_\nn$ defined by~$\theta(\xx,\yy) = -\delta_{\pp, \xx}\delta_{\qq, \yy}$. Then \eqref{E:Der3} shows that $\phi_{\pp,\qq, \nn}$ is equal to~$\der^3\theta$. Hence $\phi_{\pp,\qq, \nn}$ belongs to~$\BB^3(\AA_\nn)$.
\end{proof}

Our next goal is to choose a subfamily of coboundaries from Lemma~\ref{L:Delta3} which would form a basis of~$\BB^3(\AA_\nn)$, completed into abasis of~$\ZZ^3(\AA_\nn)$ when a constant cocycle is added. 

\begin{prop}
\label{P:Basis3}
For every~$\nn$, the group~$\ZZ^3(\AA_\nn)$ is a free abelian group of rank $2^{\nn}(2^\nn-1)+1$,  with a basis consisting of the coboundaries $\phi_{\pp,\qq, \nn}$ with $\pp \neq \ppn$ and the constant $3$-cocycle $\const_\nn$ with value~$1$. Omitting the constant cocycle, one  obtains  a basis of $\BB^3(\AA_\nn)$. Moreover, $\HH^3(\AA_\nn) \cong \ZZZZ$ holds. 
\end{prop}

As in the case of $2$-cocycles, the proof is based on several auxiliary results.

\begin{lemm}
\label{L:LastRow3}
Assume that $\phi$ is a $3$-cocycle for~$\AA_\nn$. Then, for $1 \le \xx \le 2^\nn$, we have $\phi(\xx,2^\nn,2^\nn)= \phi(\ppn, 2^\nn,2^\nn)$. 
\end{lemm}
 
\begin{proof}
 As before, we write~$\pn$ for~$2^\nn{-}1$  and $\op$ for~$\OP\nn$. Applying~\eqref{E:Cocycle4} to $(\pn, \xx, 2^\nn, 2^\nn)$  yields
\begin{align*}
\phi( \pn \op \xx ,  \pn \op 2^\nn,  \pn \op 2^\nn) + &\phi( \pn, \xx , 2^\nn \op 2^\nn) + \phi( \pn, 2^\nn, 2^\nn)\\ 
&= \phi( \pn, \xx  \op 2^\nn, \xx  \op 2^\nn) + \phi(\xx , 2^\nn, 2^\nn) + \phi( \pn, \xx , 2^\nn).
\end{align*}
 Using the relations $\pn \op \xx = 2^\nn$ and $\xx \op 2^\nn = 2^\nn$, four terms disappear and it just remains $\phi(2^\nn, 2^\nn,2^\nn)=\phi(\xx , 2^\nn, 2^\nn)$. 
\end{proof} 

\begin{lemm}
\label{L:LastRow3'}
Assume that $\phi$ is a $3$-cocycle for~$\AA_\nn$. Then, for $1 \le \zz \le 2^\nn$, we have $\phi(\ppn, \ppn, \zz)= \phi(\ppn, \ppn, \ppn)$.
\end{lemm}
 
\begin{proof}
 As usual, put $\pn = \ppn$.  Applying~\eqref{E:Cocycle4} to $(\pn, \pn, \pn, \zz)$ yields 
\begin{align*}
\phi( \pn \op \pn,  \pn \op \pn,  \pn \op \zz) + &\phi( \pn, \pn, \pn \op \zz) + \phi( \pn, \pn, \zz)\\ 
&= \phi( \pn, \pn \op \pn, \pn \op \zz) + \phi(\pn, \pn, \zz) + \phi(\pn, \pn, \zz),
\end{align*}
 leaving 
\begin{align*}
\phi( 2^\nn, 2^\nn,  2^\nn) + \phi( \pn, \pn, 2^\nn) &= \phi( \pn, 2^\nn, 2^\nn) + \phi(\pn, \pn, \zz).
\end{align*}
According to Lemma~\ref{L:LastRow3}, one has $\phi( 2^\nn, 2^\nn,  2^\nn) = \phi( \pn, 2^\nn,  2^\nn)$, and  we deduce  $\phi(\pn, \pn,\zz)=\phi(\pn, \pn, 2^\nn)$.
\end{proof} 

In other words, the values of a $3$-cocycle  for~$\AA_\nn$  on  the  triples $(\xx,2^\nn,2^\nn)$ and $(\ppn, \ppn, \zz)$  do not depend  on~$\xx$ and~$\zz$.  We now establish that a $3$-cocycle that vanishes on triples starting with~$\ppn$ must be trivial.

\begin{lemm}
\label{L:FirstColumn3}
Assume that $\phi$ is a $3$-cocycle for~$\AA_\nn$ and $\phi(\ppn, \yy,\zz) = 0$ holds for all~$\yy, \zz$  with $1 \le \yy, \zz \le 2^\nn$.  Then $\phi$ is the zero cocycle.
\end{lemm}

\begin{proof}
 Applying~\eqref{E:Cocycle4} to $(\pn, \xx, \yy, \zz)$ (with $\pn = \ppn$)  yields
\begin{align*}
\label{E:Cocycle4}
\phi(\pn \op \xx, \pn \op \yy, \pn \op \zz) + &\phi(\pn, \xx, \yy \op \zz) + \phi(\pn, \yy, \zz)\\ 
&= \phi(\pn, \xx \op \yy, \xx \op \zz) + \phi(\xx, \yy, \zz) + \phi(\pn, \xx, \zz),
\end{align*}
hence $\phi(2^\nn, 2^\nn, 2^\nn) = \phi(\xx, \yy, \zz)$, owing to~\eqref{E:LastRows} and the assumption on~$\phi$. Lemma~\ref{L:LastColumn2} gives $\phi(2^\nn, 2^\nn, 2^\nn) = \phi(\pn,2^\nn, 2^\nn)$, which is zero, again by the assumption on~$\phi$. One concludes that $\phi$ is zero on~$\AA_\nn \times \AA_\nn \times \AA_\nn$.
\end{proof}

The argument can now be completed. 
 
\begin{proof}[Proof of Proposition~\ref{P:Basis3}]
 We once more put $\pn = \ppn$.  We start with evaluating the $3$-coboundaries $\phi_{\pp,\qq, \nn}$ on triples $(\pn,\yy,\zz)$.By definition,  we have
\begin{align*} 
\phi_{\pp,\qq, \nn}(\pn, \yy, \zz) &= \delta_{\pp, \yy}\delta_{\qq, \zz} - \delta_{\pp,\pn \op \yy}\delta_{\qq, \pn \op \zz} - \delta_{\pp, \pn}\delta_{\qq, \zz} + \delta_{\pp,\pn}\delta_{\qq, \yy \op \zz}\\
&= \delta_{\pp, \yy}\delta_{\qq, \zz} - \delta_{\pp,2^\nn}\delta_{\qq, 2^\nn} - \delta_{\pp, \pn}\delta_{\qq, \zz} + \delta_{\pp,\pn}\delta_{\qq, \yy \op \zz}.
\end{align*}
For~$\pp <\pn$, most terms vanish, and  there remains 
\begin{align*}
\phi_{\pp,\qq, \nn}(\pn, \yy, \zz) &= \delta_{\pp, \yy}\delta_{\qq, \zz}.
\end{align*}
For~$\pp = 2^\nn$,  we find 
\begin{align*}
\phi_{2^\nn,\qq, \nn}(\pn, \yy, \zz) &= \delta_{2^\nn, \yy}\delta_{\qq, \zz} - \delta_{\qq, 2^\nn},
\end{align*}
which again simplifies  into  $\delta_{\pp, \yy}\delta_{\qq, \zz}$ for~$\qq \neq 2^\nn$. In order to treat the case $\qq = 2^\nn$ in a similar way,  we define for $1 \le \pp, \qq \le 2^\nn$ and $\pp \not= \pn$ new cocycles $\phi'_{\pp,\qq, \nn}$ by
\begin{equation*}
\phi'_{\pp,\qq, \nn} = \begin{cases}
\phi_{\pp, \qq, \nn}&\mbox{for $(\pp, \qq) \not= (2^\nn, 2^\nn)$},\\
\phi_{2^\nn,2^\nn, \nn}+ \const_\nn &\mbox{for $(\pp, \qq) = (2^\nn, 2^\nn)$}.\end{cases}
\end{equation*}
 Then one easily checks that the equality 
\begin{equation}
\label{E:ophi}
\phi'_{\pp,\qq, \nn}(\pn, \yy, \zz) = \delta_{\pp, \yy}\delta_{\qq, \zz}
\end{equation}
 is valid for all~$\pp, \qq$ with $\pp \not= \pn$.

Now take an arbitrary $3$-cocycle~$\phi$. Trying to approximate it with the cocycles we are interested in, put
\begin{equation*}
\VR(4,5)\phit = \smash{\sum_{\pp \neq \pn}\sum_{\qq}} (\phi(\pn, \pp, \qq) -\phi(\pn, \pn,\pn))\phi' _{\pp,\qq, \nn} + \phi(\pn, \pn,\pn) \const_\nn.
\end{equation*}
 Using~\eqref{E:ophi}, we see that the cocycle $\phi-\phit$ vanishes on  all triples $(\pn, \yy, \zz)$ with $\yy \not= \pn$ and on~$(\pn, \pn, \pn)$. On the other hand,  Lemma~\ref{L:LastRow3'} asserts that  $\phi-\phit$ vanishes on all triples $(\pn, \pn, \zz)$. Merging the results, we deduce that $\phi-\phit$  vanishes on all triples $(\pn, \yy, \zz)$, hence, according to Lemma~\ref{L:FirstColumn3}, it is the zero cocycle. This proves that the coboundaries~$\phi_{\pp,\qq, \nn}$ with~$\pp \neq \pn$  together with  the constant $3$-cocycle~$\const_\nn$ generate~$\ZZ^3(\AA_\nn)$. 

 The freeness of the above family is established exactly as in the case of $2$-cocycles, and so are the assertions for~$\BB^3(\AA_\nn)$. The only difficulty consists in showing that the coboundaries~$\phi_{\pp,\qq, \nn}$ with~$\pp \neq \pn$ generate~$\BB^3(\AA_\nn)$. For this, we observe that all coboundaries vanish on the triple $(\pn,2^\nn,2^\nn)$, whereas the constant $3$-cocycle does not.
\end{proof}

We conclude  with  some observations about the values taken by the cocycles~$\phi_{\pp,\qq, \nn}$ and, more generally, by the cocycles occurring in a basis of~$\ZZ^3(\AA_\nn)$. 

\begin{prop}
\label{P:Not2}
For all~$\nn, \pp,\qq$, the $3$-cocycle $\phi_{\pp,\qq, \nn}$ evaluated on a triple from~$\AA_\nn^3$ can only take the value~$0$, $+1$, or $-1$. 
\end{prop}

\begin{proof}
 First, by the definition of~\eqref{E:phi3}, the value of~$\phi_{\pp,\qq, \nn}(\xx, \yy, \zz)$ for $\xx, \yy, \zz$ in~$\AA_\nn$ must lie in $\{0, \pm 1, \pm 2 \}$. We shall prove that $\pm2$ is impossible. 
 
Indeed, $\phi_{\pp,\qq, \nn}(\xx, \yy, \zz) = 2$ would  require 
\begin{align*}
\delta_{\pp, \yy}\delta_{\qq, \zz} = \delta_{\pp,\xx}\delta_{\qq, \yy \op \zz} =1 \mbox{\quad and \quad} \delta_{\pp,\xx \op \yy}\delta_{\qq, \xx \op \zz} = \delta_{\pp, \xx}\delta_{\qq, \zz} =0.
\end{align*}
 The first two equalities imply  $\yy = \pp = \xx$  and $\zz = \qq$, whence $\delta_{\pp, \xx}\delta_{\qq, \zz} = 1$, which contradicts the last equality.

Similarly,  $\phi_{\pp,\qq, \nn}(\xx, \yy, \zz) = -2$ would  require
\begin{align*}
\delta_{\pp, \yy}\delta_{\qq, \zz} = \delta_{\pp,\xx}\delta_{\qq, \yy \op \zz} =0  \mbox{\quad and \quad}  \delta_{\pp,\xx \op \yy}\delta_{\qq, \xx \op \zz} = \delta_{\pp, \xx}\delta_{\qq, \zz} =1.
\end{align*}
 The last two equalities imply  $\xx \op \yy = \pp =\xx$ and $\zz = \qq$, which, by~\eqref{E:LeftMultiples} and~\eqref{E:LastRows},  can occur only for $\xx = \yy = 2^\nn$, in which case we deduce $\delta_{\pp, \yy}\delta_{\qq, \zz} = 1$, which contradicts the first equality. 
\end{proof}

However, contrary to the case of~$\ZZ^2(\AA_\nn)$, we cannot expect to find for every~$\nn$ a basis of~$\ZZ^3(\AA_\nn)$ consisting of $\{0, 1\}$-valued cocycles.

\begin{prop}
\label{P:NotZeroOne}
There is no basis of $\ZZ^3(\AA_1)$ consisting of cocycles whose values on~$\AA_1\times\AA_1\times\AA_1$ lie in~$\{0,1\}$.
\end{prop}

\begin{proof}
Put $\phit_{2,2,1}=-\phi_{2,2,1}-\phi_{2,1,1}$. By Proposition~\ref{P:Basis3}, $\{\phi_{2,1,1}, \phi_{2,2,1}, \const_1\}$ is a basis of~$\ZZ^3(\AA_1)$, hence so is $\{\phi_{2,1,1}, \phit_{2,2,1}, \const_1\}$. Now, direct computations give  
\begin{equation}
\label{E:NotZeroOne}
\phi_{2,1,1} \begin{cases}
=1 &\mbox{on $(1, 2, 1)$},\\
=-1 &\mbox{on $(2, 1, 1)$},\\
=0 &\mbox{elsewhere}, \end{cases}
\quad
\phit_{2,2,1} \begin{cases}
= 1 &\mbox{on $(1, 1, 1)$ and $(1, 1, 2)$},\\
= 0 &\mbox{elsewhere}. \end{cases}
\end{equation}
Let $\{\zeta_1, \zeta_2, \zeta_3\}$ be an arbitrary basis of~$\ZZ^3(\AA_1)$. For every~$\ii$, there exist integers~$\lambda_\ii, \mu_\ii, \nu_\ii$ satisfying $\zeta_\ii = \lambda_\ii \phi_{2,1,1} + \mu_\ii \phit _{2,2,1} + \nu_\ii \const_1$, and $\lambda_\ii \not= 0$ holds for at least one~$\ii$. For such an index~$\ii$, \eqref{E:NotZeroOne} gives $\zeta_\ii(1,2,1) - \zeta_\ii(2,1,1) = 2\lambda_\ii$, which is possible only if $\zeta_\ii(1,2,1)$ or $\zeta_\ii(2,1,1)$ does not lie in~$\{0, 1\}$.
\end{proof}

We shall not go further here. Most steps in the proof of Proposition~\ref{P:Basis3} can be extended to the case of $\kk$-cocycles with $\kk \ge 4$. However, finding minimal vanishing conditions for $\kk$-cocycles seems to be a bottleneck. Extending Lemma~\ref{L:FirstColumn3}, one can check that a $\kk$-cocycle for~$\AA_\nn$ that is zero on all $\kk$-tuples starting with~$\ppn$ must be trivial, but this condition is not minimal: for instance, in the case $\kk = 3$, for a cocycle to be zero on the triples $(\ppn, \yy, \zz)$ with $\yy \not= \ppn$ and on $(\ppn,\ppn,\ppn)$ is sufficient to deduce that it is zero everywhere, and this is what results in the rank $2^{2\nn} - 2^\nn + 1$ for~$\ZZ^3(\AA_\nn)$. A similar analysis is possible for $\kk = 4$, resulting in the rank $2^{3\nn} -  2^{2\nn} +   2^\nn$ for~$\ZZ^4(\AA_\nn)$, but the general case remains unclear at the moment.


\end{document}